\documentclass[11pt,reqno]{amsart}

\usepackage[T1]{fontenc}
\usepackage{lmodern}
\usepackage{microtype}
\usepackage{amsmath,amssymb,amsthm,mathtools}
\usepackage{booktabs}
\usepackage{array}
\usepackage{enumitem}
\usepackage{geometry}
\usepackage{placeins}
\usepackage{tikz}
\usetikzlibrary{positioning,calc,arrows.meta}
\usepackage[colorlinks=true,linkcolor=blue,citecolor=blue,urlcolor=blue]{hyperref}
\usepackage{aliascnt}
\usepackage[nameinlink,noabbrev]{cleveref}

\newtheorem{theorem}{Theorem}[section]
\newaliascnt{proposition}{theorem}
\newtheorem{proposition}[proposition]{Proposition}
\aliascntresetthe{proposition}
\newaliascnt{lemma}{theorem}
\newtheorem{lemma}[lemma]{Lemma}
\aliascntresetthe{lemma}
\newaliascnt{corollary}{theorem}
\newtheorem{corollary}[corollary]{Corollary}
\aliascntresetthe{corollary}

\crefname{theorem}{Theorem}{Theorems}
\Crefname{theorem}{Theorem}{Theorems}
\crefname{proposition}{Proposition}{Propositions}
\Crefname{proposition}{Proposition}{Propositions}
\crefname{lemma}{Lemma}{Lemmas}
\Crefname{lemma}{Lemma}{Lemmas}
\crefname{corollary}{Corollary}{Corollaries}
\Crefname{corollary}{Corollary}{Corollaries}
\crefname{section}{Section}{Sections}
\Crefname{section}{Section}{Sections}
\crefname{appendix}{Appendix}{Appendices}
\Crefname{appendix}{Appendix}{Appendices}
\crefname{equation}{Equation}{Equations}
\Crefname{equation}{Equation}{Equations}
\crefname{figure}{Figure}{Figures}
\Crefname{figure}{Figure}{Figures}
\crefname{table}{Table}{Tables}
\Crefname{table}{Table}{Tables}

\newcommand{\Pp}{\mathbb P}
\newcommand{\Ee}{\mathbb E}
\newcommand{\Rr}{\mathbb R}
\newcommand{\Zz}{\mathbb Z}
\newcommand{\one}{\mathbf 1}
\newcommand{\Bin}{\operatorname{Bin}}
\newcommand{\Bern}{\operatorname{Bernoulli}}
\newcommand{\PCS}{\operatorname{PCS}}
\newcommand{\argmin}{\operatorname*{arg\,min}}
\newcommand{\KL}{D}
\newcommand{\pc}{p_{\mathrm c}}
\newcommand{\taud}{\tau_\delta}
\newcommand{\Gam}{\Gamma_\delta}
\newcommand{\Cd}{C_\delta}
\newcommand{\Jn}{J_\delta}
\newcommand{\Kn}{K_\delta}
\newcommand{\ThetaPZ}{\Theta_{k,t,\delta}}

\title[Least-favorable location for binomial top-$t$ selection]{Least-Favorable Location for Binomial Top-$t$ Selection}
\author{Yinhao Wu}
\address{School of Mathematics, Shanghai University of Finance and Economics, Shanghai 200433, China}
\email{yinhaowu@stu.sufe.edu.cn}
\author{Pinyuen Chen\textsuperscript{*}}
\address{Department of Mathematics, Syracuse University, Syracuse, NY 13244, USA}
\email{pinchen@syr.edu}
\thanks{\textsuperscript{*}Corresponding author. E-mail address: pinchen@syr.edu}
\date{\today}

\begin{document}

\begin{abstract}
Consider $k$ independent Bernoulli populations, each sampled $n$ times, and select the $t$ populations with the largest success counts, breaking ties uniformly. Classical monotonicity reduces the worst case over the preference zone with separation $\delta$ to the slippage family with levels $p$ and $p+\delta$, leaving only its absolute location $p\in[0,1-\delta]$ undetermined. A Gaussian approximation suggests the symmetric center $p_{\mathrm c}=(1-\delta)/2$, and the exact two-population problem is uniquely centered there for every $n\ge2$. For fixed $k,t$ and $\delta\in(0,1)$, we prove that exact eventual centering holds precisely when $k=2t$. When $k\ne2t$, the least-favorable location $p_{n,k,t}^*$ satisfies
\[
 p_{n,k,t}^*-p_{\mathrm c}
 =(k-2t)C_\delta n^{-1/2}e^{-n\Gamma_\delta}\{1+o(1)\},
\]
where $C_\delta$ and $\Gamma_\delta$ are explicit and positive. In either case, the least-favorable location is eventually unique. The proof writes incorrect selection as a union of pairwise misrankings and applies inclusion--exclusion, yielding a bipartite graph expansion. A single misranking has its exact maximum at the symmetric center and determines the central curvature; two-edge intersections sharing one population determine the central slope through their multiplicity imbalance; all remaining graphs have higher large-deviation rates.
\end{abstract}

\maketitle \enlargethispage{7pt}

\section{Introduction}

In the least-favorable-configuration (LFC) problem for fixed-sample binomial top-$t$ selection with uniform tie breaking, the classical reduction identifies a two-level slippage family but, except in special cases, does not determine the minimizing location along it. Specifically, consider $k$ independent Bernoulli populations, each sampled $n$ times, and select the $t$ populations with the largest success counts, using uniform randomization when ties cross the selection boundary. For ordered parameters $\theta_{[1]}\le\cdots\le\theta_{[k]}$, define the preference zone by
\[
 \Theta_{k,t,\delta}
 =\{\boldsymbol\theta\in[0,1]^k:
   \theta_{[k-t+1]}-\theta_{[k-t]}\ge\delta\}.
\]
Classical monotonicity~\cite{Bofinger1976} reduces the worst probability of correct selection (PCS) over this space to
\[
 \inf_{\boldsymbol\theta\in\Theta_{k,t,\delta}}
 \PCS_n(\boldsymbol\theta)
 =\min_{0\le p\le1-\delta}P_{n,k,t,\delta}(p),
\]
where $P_{n,k,t,\delta}(p)$ is the PCS when $k-t$ populations have success probability $p$ and the other $t$ have success probability $p+\delta$. \Cref{prop:full-reduction} also identifies the equality case: if the scalar minimum is attained uniquely at an interior point, every full-space LFC has exactly these two levels. The remaining problem is therefore to locate this scalar minimizer and prove its uniqueness. Whenever it is unique, denote it by $p_{n,k,t}^*$.

The Gaussian approximation singles out a natural candidate. For one low count $X\sim\Bin(n,p)$ and one high count $Y\sim\Bin(n,p+\delta)$, the mean gap is $n\delta$ and the variance of $Y-X$ is $nV_\delta(p)$, where
\[
 V_\delta(p)
 =p(1-p)+(p+\delta)(1-p-\delta)
 =\frac{1-\delta^2}{2}-2(p-p_{\mathrm c})^2,
 \qquad p_{\mathrm c}=\frac{1-\delta}{2}.
\]
The Gaussian signal-to-noise ratio is smallest at $p_{\mathrm c}$, making a low--high comparison hardest there. Sobel and Huyett~\cite{SobelHuyett1957} obtained this centered large-sample limit for select-one binomial selection. In the finite-sample select-one problem, Buzaianu and Chen~\cite{BuzaianuChen2005} used exact derivations to obtain the LFCs for certain binomial procedures and established the equal-low slippage form at a fixed upper success probability, leaving the common absolute level to be optimized. For two Bernoulli populations, Eaton and Gleser~\cite{EatonGleser1989} proved that the centered pair is least favorable under uniform random tie breaking whenever $n\ge\max\{4,\delta^{-1}\}$. More recently, Wang~\cite{Wang2023} proved that the LFC of a related exact binomial top-$t$ procedure with a control has two experimental levels for each fixed control success probability, and left its absolute location open when the control probability is unknown. Neither Zhang and Chen~\cite{ZhangChen2025} nor Yin, Buzaianu, Chen, and Hsu~\cite{YinBuzaianuChenHsu2026} was successful in obtaining LFC's for their binomial subset selection procedures. For the classical top-$t$ rule studied here, the finite-sample location has not been characterized theoretically when $k\ge3$.

Intuitively, symmetry divides the problem into two natural cases. Complementing all Bernoulli outcomes gives
\[
 P_{n,k,t,\delta}(p)=P_{n,k,k-t,\delta}(1-\delta-p).
\]
When $k=2t$, this is a symmetry of the same objective, so uniqueness of the global minimizer would force exact centering. When $k\ne2t$, complementation maps the objective to the distinct problem with the two group sizes exchanged and does not locate either minimizer. The symmetry therefore reduces the balanced problem to eventual uniqueness, while leaving the unbalanced location to be determined.

The main theorem confirms this distinction: for fixed $k,t,\delta$, there is $n_0=n_0(k,t,\delta)$ such that $P_{n,k,t,\delta}$ has the unique minimizer $p_{n,k,t}^*$ for $n\ge n_0$. If $k=2t$, then $p_{n,k,t}^*=p_{\mathrm c}$; if $k\ne2t$, then
\[
 p_{n,k,t}^*
 =p_{\mathrm c}+(k-2t)C_\delta n^{-1/2}e^{-n\Gamma_\delta}\{1+o(1)\}.
\]
Here $C_\delta$ and $\Gamma_\delta$ are explicit and positive. When $k\ne2t$, the leading coefficient is nonzero, so the exact optimizer is off center for every sufficiently large $n$: it approaches $p_{\mathrm c}$ from the right when $k>2t$ and from the left when $k<2t$. Its distance from the center is smaller than every algebraic power of $n^{-1}$; the PCS at $p_{\mathrm c}$ likewise differs from the worst-case PCS by an exponentially small amount.

Our proof applies inclusion--exclusion to the pairwise low--high misranking events, organizing their intersections as bipartite comparison graphs. The single-edge probability is exactly even about $p_{\mathrm c}$, so it supplies the leading curvature but no central slope. The first nonzero slope comes from the two orientations of a two-edge wedge. Their multiplicities differ by a factor proportional to $k-2t$, while all other graph terms have strictly larger rates. Put $Q_{n,k,t,\delta}(p)=1-P_{n,k,t,\delta}(p)$. After localization and uniqueness have been established, the unbalanced displacement satisfies
\begin{equation}
 p_{n,k,t}^*-p_{\mathrm c}
 \sim-\frac{Q_{n,k,t,\delta}'(p_{\mathrm c})}
 {Q_{n,k,t,\delta}''(p_{\mathrm c})}.
 \label{eq:intro-quotient}
\end{equation}
The denominator is the edge curvature, with large-deviation rate $J_\delta$; the numerator is the wedge slope, with rate $K_\delta>J_\delta$. Their quotient has order $(k-2t)n^{-1/2}e^{-n(K_\delta-J_\delta)}$, so $\Gamma_\delta=K_\delta-J_\delta$; the ratio of the two derivative prefactors gives $C_\delta$.

At the leading large-deviation scale, Glynn and Juneja~\cite{GlynnJuneja2004} obtained the same center in the equal-allocation Bernoulli indifference-zone model as part of their rate-based analysis of sampling allocation. In the present problem, that scale identifies the limiting center and PCS scale; determining the exact finite-sample location requires the exponentially smaller scale. Our use of inclusion--exclusion is inspired by Shi, Peng, and Tuffin~\cite{ShiPengTuffin2024}, who show that simultaneous incorrect comparisons omitted by pairwise expansions can affect refined PCS approximations and finite-budget allocations. Here shared-population wedges are exponentially smaller than the edge contribution to the PCS, yet determine the leading central slope because the edge slope vanishes exactly.

The paper is organized as follows. \Cref{sec:problem-main} gives the full-space reduction and states the main result. \Cref{sec:edge-wedge} derives $Q_n'(\pc)$ and $Q_n''(\pc)$ from the comparison graph, and \Cref{sec:proof} turns these local estimates into localization, uniqueness, and the optimizer expansion. \Cref{sec:numerical,sec:conclusion} give numerical illustrations and the even-sample centering conjecture. \Cref{app:finite-sample} proves the full-space reduction and exact single-edge centering; \Cref{app:motif-calculations,app:graph-separation} contain the sharp lattice asymptotics and the graph-rate and localization estimates.

\section{Problem and main result}\label{sec:problem-main}

\subsection{Model and classical reduction}

Let $1\le t<k$ and let $\boldsymbol\theta=(\theta_1,\ldots,\theta_k)\in[0,1]^k$. Population $i$ produces independent observations $B_{i1},\ldots,B_{in}\sim\Bern(\theta_i)$ and count $S_i=\sum_{m=1}^nB_{im}$. Independent variables $U_i\sim\mathrm{Unif}(0,1)$ are used to order the augmented scores $\widetilde S_i=(S_i,U_i)$ lexicographically. The rule selects the $t$ largest augmented scores.

Write $\theta_{[1]}\le\cdots\le\theta_{[k]}$ for the ordered parameters and define
\begin{equation*}
 \ThetaPZ
 =\left\{\boldsymbol\theta\in[0,1]^k:
 \theta_{[k-t+1]}-\theta_{[k-t]}\ge\delta\right\}.
\end{equation*}
For $\boldsymbol\theta\in\ThetaPZ$, the top $t$ populations are separated from the other $k-t$ populations by at least $\delta$ and therefore form a well-defined set. Let $\PCS_n(\boldsymbol\theta)$ denote the probability that the rule selects exactly that set. We use the following finite-sample form of the classical reduction.

\begin{proposition}[Full-space reduction and its interior equality case]\label{prop:full-reduction}
For $p\in[0,1-\delta]$, let $P_{n,k,t,\delta}(p)$ be the PCS for $k-t$ probabilities equal to $p$ and $t$ probabilities equal to $p+\delta$. For each $\boldsymbol\theta\in\ThetaPZ$, define
\[
 p(\boldsymbol\theta)=\theta_{[k-t+1]}-\delta\in[0,1-\delta].
\]
Then
\begin{equation}
 \PCS_n(\boldsymbol\theta)
 \ge P_{n,k,t,\delta}\bigl(p(\boldsymbol\theta)\bigr).
 \label{eq:pointwise-reduction}
\end{equation}
If $p(\boldsymbol\theta)\in(0,1-\delta)$, equality holds if and only if
\[
 \theta_{[1]}=\cdots=\theta_{[k-t]}=p(\boldsymbol\theta),
 \qquad
 \theta_{[k-t+1]}=\cdots=\theta_{[k]}
 =p(\boldsymbol\theta)+\delta.
\]
Consequently, for every $n,k,t$ and $\delta\in(0,1)$,
\begin{equation}
 \inf_{\boldsymbol\theta\in\ThetaPZ}\PCS_n(\boldsymbol\theta)
 =\min_{0\le p\le1-\delta}P_{n,k,t,\delta}(p).
 \label{eq:full-reduction}
\end{equation}
\end{proposition}

Thus \Cref{prop:full-reduction} reduces the LFC problem to minimizing the scalar function $P_{n,k,t,\delta}$. Its equality case recovers the full-space minimizers whenever the scalar minimizer is interior and unique. The non-strict general-$t$ reduction is classical \cite{Bofinger1976}. For $t=1$, Sobel and Huyett~\cite{SobelHuyett1957} proved strict monotonicity in each best--inferior parameter gap. \Cref{app:full-reduction-proof} proves the general-$t$ strict inequality needed here.

\subsection{The two-level objective}

Fix $k\ge2$, $1\le t<k$, a sample size $n\ge1$, and a gap $\delta\in(0,1)$. For $p\in[0,1-\delta]$, write $q=p+\delta$ and let
\[
 X_1,\ldots,X_{k-t}\stackrel{\mathrm{iid}}{\sim}\Bin(n,p),
 \qquad
 Y_1,\ldots,Y_t\stackrel{\mathrm{iid}}{\sim}\Bin(n,q),
\]
with all variables independent. The $X_i$ are the low populations and the $Y_j$ are the high populations. Exactly $t$ populations are selected.

Uniform random tie breaking is represented by independent auxiliary variables
\[
 U_1,\ldots,U_{k-t},V_1,\ldots,V_t
 \stackrel{\mathrm{iid}}{\sim}\mathrm{Unif}(0,1),
\]
independent of the binomial counts. We order augmented scores
\[
 \widetilde X_i=(X_i,U_i),\qquad \widetilde Y_j=(Y_j,V_j)
\]
lexicographically. Selecting the $t$ largest augmented scores is equivalent to choosing uniformly among the admissible selections whenever a count tie straddles the selection boundary. A correct selection occurs precisely when
\[
 \max_{1\le i\le k-t}\widetilde X_i
 <
 \min_{1\le j\le t}\widetilde Y_j.
\]
Define
\begin{equation*}
 P_{n,k,t,\delta}(p)
 =\Pp\!\left(\max_i\widetilde X_i<\min_j\widetilde Y_j\right),
 \qquad
 Q_{n,k,t,\delta}(p)=1-P_{n,k,t,\delta}(p).
\end{equation*}
When $k,t$, and $\delta$ are fixed, we write $P_n=P_{n,k,t,\delta}$ and $Q_n=Q_{n,k,t,\delta}$. The minimizing set of $P_{n,k,t,\delta}$ is
\begin{equation*}
 \mathcal M_{n,k,t,\delta}
 =\argmin_{0\le u\le1-\delta}P_{n,k,t,\delta}(u).
\end{equation*}
Whenever this set is a singleton, denote its element by $p_{n,k,t}^*$; its dependence on $\delta$ is suppressed from the notation. For each count configuration, the conditional probability of correct selection is fixed by its tie pattern. Hence $P_{n,k,t,\delta}(p)$ is a finite sum of polynomial terms in $p$, so it is continuous on $[0,1-\delta]$ and $\mathcal M_{n,k,t,\delta}$ is nonempty.

\subsection{Main result}

Complementing successes and failures reverses the score order and gives the following exact identity.

\begin{proposition}[Exact reflection]\label{prop:reflection}
Let $T(p)=1-\delta-p$. Then, for every $n,k,t$ and $p\in[0,1-\delta]$,
\begin{equation*}
 P_{n,k,t,\delta}(p)=P_{n,k,k-t,\delta}(T(p)).
\end{equation*}
Consequently, when $k=2t$,
\begin{equation*}
 P_{n,k,t,\delta}(\pc+u)=P_{n,k,t,\delta}(\pc-u),
 \qquad \pc=\frac{1-\delta}{2},
\end{equation*}
whenever both arguments belong to $[0,1-\delta]$.
\end{proposition}

\begin{proof}
Map every augmented score $(z,u)$ to $(n-z,1-u)$. This reverses lexicographic order. If $X\sim\Bin(n,p)$, then $n-X\sim\Bin(n,1-p)=\Bin(n,T(p)+\delta)$, so a complemented low score is a high score in the reflected model. Likewise, if $Y\sim\Bin(n,p+\delta)$, then $n-Y\sim\Bin(n,1-p-\delta)=\Bin(n,T(p))$, so a complemented high score is a low score. The correct-selection event is thereby mapped to its reflected counterpart with $t$ and $k-t$ interchanged. The balanced identity follows from $T(\pc+u)=\pc-u$.
\end{proof}

When $k=2t$, reflection makes $\pc$ stationary; global centering will follow from the eventual uniqueness proved below. To state the constants in the location expansion, put
\begin{equation*}
 \taud=\left(\frac{1-\delta}{1+\delta}\right)^{1/3},
\end{equation*}
and define
\begin{equation}
 \Gam=\log\frac{4(1-\taud+\taud^2)}{(1+\taud)^2},
 \qquad
 \Cd=\frac{\taud^2}{3\sqrt{3\pi}(1+\taud)^2(1-\taud+\taud^2)}.
 \label{eq:C-closed}
\end{equation}
Both constants are positive. Here $\taud$ shortens the closed forms. The exponent gap $\Gam$ and the prefactor $\Cd$ arise, respectively, from the wedge--edge rate difference and the ratio of the wedge slope to the edge curvature; see \Cref{sec:edge-wedge,sec:proof}.

\begin{theorem}[Least-favorable location for top-$t$ selection]\label{thm:global-top-t}
Fix $k\ge2$, $1\le t<k$, and $\delta\in(0,1)$. There is $n_0=n_0(k,t,\delta)$ such that $\mathcal M_{n,k,t,\delta}$ is a singleton for all $n\ge n_0$. If $k=2t$, then
\begin{equation}
 p_{n,k,t}^*=\frac{1-\delta}{2}
 \qquad(n\ge n_0).
 \label{eq:balanced-select-t}
\end{equation}
If $k\ne2t$, then
\begin{equation}
 p_{n,k,t}^*
 =\frac{1-\delta}{2}
 +(k-2t)\Cd n^{-1/2}e^{-n\Gam}\{1+o(1)\}.
 \label{eq:select-t-main}
\end{equation}
Thus the displacement is to the right when $k>2t$ and to the left when $k<2t$. Consequently, for every $M>0$,
\begin{equation}
 n^M|p_{n,k,t}^*-\pc|\longrightarrow0.
 \label{eq:beyond-all-orders}
\end{equation}
\end{theorem}

For two populations, the location can be determined for every sample size.

\begin{corollary}[Exact two-population location]\label{cor:two-population-location}
For every $\delta\in(0,1)$,
\begin{equation*}
 \mathcal M_{n,2,1,\delta}
 =
 \begin{cases}
 [0,1-\delta], & n=1,\\
 \{\pc\}, & n\ge2.
 \end{cases}
\end{equation*}
For every $n\ge2$, $\boldsymbol\theta$ minimizes $\PCS_n(\boldsymbol\theta)$ over $\Theta_{2,1,\delta}$ if and only if $(\theta_{[1]},\theta_{[2]})=(\pc,\pc+\delta)$. Thus one may take $n_0(2,1,\delta)=2$ in \Cref{thm:global-top-t}.
\end{corollary}

\begin{proof}
When $k=2$ and $t=1$, incorrect selection is the single tie-weighted misranking event. Write
\[
 a_n(p):=\Pp(X>Y)+\frac12\Pp(X=Y),
\]
where $X\sim\Bin(n,p)$ and $Y\sim\Bin(n,p+\delta)$ are independent. Then $P_{n,2,1,\delta}(p)=1-a_n(p)$. The calculation in \Cref{app:one-edge-centering} shows that $a_1$ is constant and that $a_n$ is uniquely maximized at $\pc$ for every $n\ge2$. This gives the stated scalar minimizing sets. For $n\ge2$, the minimizer $\pc$ is interior, and the strict inequality in \Cref{prop:full-reduction} gives the full-space conclusion.
\end{proof}

This sharpens Eaton and Gleser~\cite{EatonGleser1989}, whose centering result under uniform random tie breaking assumes $n\ge\max\{4,\delta^{-1}\}$. The balanced conclusion in \eqref{eq:balanced-select-t} extends centering from $k=2,t=1$ to every balanced top-$t$ problem $k=2t$. For general $t$, the proof yields a threshold $n_0(2t,t,\delta)$ whose dependence on $t$ enters through simultaneous low--high misorderings.

For $t=1$ and $k\ge3$, the coefficient $k-2$ is positive, so the location is eventually to the right of the reflection center. More generally, \eqref{eq:beyond-all-orders} shows that the exact off-center displacement is invisible at every algebraic order in $n^{-1}$. The next section resolves it by computing the central slope and curvature on their two exponential scales.

\section{The edge--wedge expansion}\label{sec:edge-wedge}

We henceforth maximize the incorrect-selection probability $Q_n=1-P_n$. This section computes the two quantities in \eqref{eq:intro-quotient}. We organize simultaneous low--high misorderings by a bipartite comparison graph: a single edge supplies the leading curvature, the two orientations of a connected two-edge wedge supply the first nonzero slope, and the remaining graphs are separated by their rates.

A correct top-$t$ selection means that every low augmented score is below every high augmented score. Index each pairwise error $\widetilde X_i>\widetilde Y_j$ by the edge $(i,j)$ of the complete bipartite graph $K_{k-t,t}$. Incorrect selection occurs if and only if at least one such edge is present.

For an edge $e=(i,j)$ put
\[
 E_e=E_{ij}=\{\widetilde X_i>\widetilde Y_j\}.
\]
Then
\begin{equation*}
 Q_n(p)
 =\Pp\left(\bigcup_{i=1}^{k-t}\bigcup_{j=1}^tE_{ij}\right).
\end{equation*}
Inclusion--exclusion then classifies simultaneous errors by their edge sets. Two distinct edges either share a high vertex, share a low vertex, or are disjoint, so \Cref{fig:motifs} lists every motif through second order.

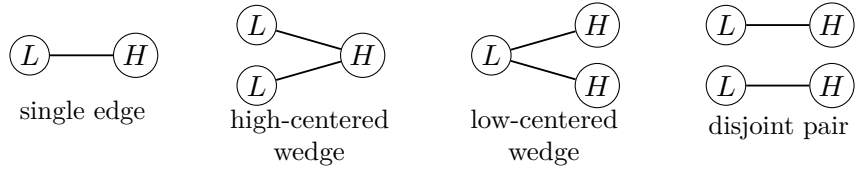
\begin{figure}[ht]
\centering
\begin{tikzpicture}[
  every node/.style={circle,draw,inner sep=1.5pt,minimum size=5mm},
  lab/.style={draw=none,rectangle,inner sep=0pt,font=\small,align=center},
  edge/.style={line width=0.7pt}
]
\node (l1) at (0,0.4) {$L$};
\node (h1) at (1.4,0.4) {$H$};
\draw[edge] (l1)--(h1);
\node[lab] at (0.7,-0.35) {single edge};
\node (l21) at (3.0,0.8) {$L$};
\node (l22) at (3.0,0.0) {$L$};
\node (h2) at (4.4,0.4) {$H$};
\draw[edge] (l21)--(h2);
\draw[edge] (l22)--(h2);
\node[lab,text width=2.5cm] at (3.7,-0.68) {high-centered\\wedge};
\node (l3) at (6.1,0.4) {$L$};
\node (h31) at (7.5,0.8) {$H$};
\node (h32) at (7.5,0.0) {$H$};
\draw[edge] (l3)--(h31);
\draw[edge] (l3)--(h32);
\node[lab,text width=2.5cm] at (6.8,-0.68) {low-centered\\wedge};
\node (l41) at (9.2,0.8) {$L$};
\node (h41) at (10.6,0.8) {$H$};
\node (l42) at (9.2,0.0) {$L$};
\node (h42) at (10.6,0.0) {$H$};
\draw[edge] (l41)--(h41);
\draw[edge] (l42)--(h42);
\node[lab] at (9.9,-0.55) {disjoint pair};
\end{tikzpicture}
\caption{The four motifs needed through second order. An edge is one
low--high misordering. A wedge consists of two errors sharing one population;
a disjoint pair consists of two independent errors.}
\label{fig:motifs}
\end{figure}

Write the one-edge probability as
\begin{equation*}
 a_n(p)=\Pp(\widetilde X>\widetilde Y)
 =\Pp(X>Y)+\frac12\Pp(X=Y),
\end{equation*}
where $X\sim\Bin(n,p)$ and $Y\sim\Bin(n,p+\delta)$ are independent. By \Cref{prop:reflection,cor:two-population-location}, $a_n$ is exactly even about $\pc$ and, for $n\ge2$, uniquely maximized there. An independent-edge approximation therefore remains centered; the first terms that can shift the optimizer are the two orientations of a connected two-edge subgraph:
\begin{equation*}
 H_n(p)=\Pp(\widetilde X_1>\widetilde Y,\widetilde X_2>\widetilde Y),
 \qquad
 L_n(p)=\Pp(\widetilde X>\widetilde Y_1,\widetilde X>\widetilde Y_2).
\end{equation*}
The letters $H$ and $L$ indicate whether the shared vertex is high or low. Complementation gives
\begin{equation}
 L_n(p)=H_n(T(p)).
 \label{eq:HL-reflection}
\end{equation}

\begin{proposition}[Bonferroni bounds and second-order graph decomposition]\label{prop:bonferroni}
Set
\[
 N_H=t\binom{k-t}{2},\qquad
 N_L=(k-t)\binom t2,\qquad
 N_D=2\binom{k-t}{2}\binom t2.
\]
Then, for every $n,p,k,t,\delta$,
\begin{equation}
 t(k-t)a_n(p)-N_HH_n(p)-N_LL_n(p)-N_Da_n(p)^2
 \le Q_n(p)\le t(k-t)a_n(p).
 \label{eq:bonferroni}
\end{equation}
Moreover, inclusion--exclusion gives the exact decomposition
\begin{equation}
 Q_n(p)
 =t(k-t)a_n(p)-N_HH_n(p)-N_LL_n(p)-N_Da_n(p)^2+\mathcal R_n(p),
 \label{eq:exact-graph-decomp}
\end{equation}
where $\mathcal R_n$ is the alternating sum over edge sets containing at least three edges.
\end{proposition}

\begin{proof}
The first-order sum has $t(k-t)$ identical terms, each equal to $a_n(p)$. Two distinct edges that share a high endpoint have intersection probability $H_n(p)$, and there are $t\binom{k-t}{2}$ such pairs. Edges sharing a low endpoint have intersection probability $L_n(p)$, and there are $(k-t)\binom t2$ such pairs. Two disjoint edges involve four independent augmented scores, so their intersection probability is $a_n(p)^2$; choosing two low and two high vertices leaves two disjoint matchings, giving $2\binom{k-t}{2}\binom t2$ pairs. Bonferroni's inequalities give \eqref{eq:bonferroni}, and retaining the remaining inclusion--exclusion terms gives \eqref{eq:exact-graph-decomp}.
\end{proof}

Both the single-edge term $a_n$ and the disjoint-pair term $a_n^2$ are exactly even about $\pc$, so neither contributes to the central slope. Among the second-order terms, only the two wedge orientations can contribute to that slope. Their multiplicity difference is
\begin{equation}
 N_H-N_L
 =t\binom{k-t}{2}-(k-t)\binom t2
 =\frac{t(k-t)}2(k-2t).
 \label{eq:wedge-count-difference}
\end{equation}
This multiplicity difference supplies the factor $k-2t$ in \Cref{thm:global-top-t}. It remains to compute the edge curvature and wedge slope sharply and to show that every other graph term lies on a smaller probability scale.

\subsection{The leading edge and wedge terms}\label{sec:motifs}

Rate comparisons identify the dominant motifs, but the optimizer displacement depends on a wedge derivative divided by an edge second derivative. We therefore need their polynomial prefactors and uniform control through two derivatives. For a fixed neighborhood $U$ of $\pc$, the notation $o_{C^2(U)}(1)$ denotes a term that converges uniformly to zero on $U$ together with its first two derivatives. \Cref{app:single-edge-asymptotics,app:wedge-asymptotics} prove the required uniform edge and wedge asymptotics.

For $q=p+\delta$, the one-edge large-deviation rate is
\begin{equation}
 I_1(p)
 =-2\log\left(\sqrt{pq}+\sqrt{(1-p)(1-q)}\right).
 \label{eq:I1}
\end{equation}
It is the minimum Bernoulli relative-entropy cost of forcing a low and a high empirical proportion to cross. The two deviations enter with equal weight, in agreement with the exact reflection symmetry of the edge probability. Its closed form is used here for differentiation; the variational representation in \Cref{app:single-edge-asymptotics} connects it to the general graph rates in \Cref{app:graph-rate-details}. The following lemma gives the rate properties and uniform strong expansion used in the optimizer quotient.

\begin{lemma}[Single-edge expansion]\label{lem:one-edge-rate}
The rate $I_1$ is invariant under $p\mapsto T(p)$ and has its unique minimum at $\pc$. At the center,
\begin{equation}
 \Jn:=I_1(\pc)=-\log(1-\delta^2),
 \qquad
 I_1''(\pc)=\frac{8\delta^2}{(1-\delta^2)^2}.
 \label{eq:J-and-curvature}
\end{equation}
On some closed interval $U\subset(0,1-\delta)$ centered at $\pc$, the single-edge probability has the strong expansion
\begin{equation}
 a_n(p)=n^{-1/2}e^{-nI_1(p)}
 \{A_1(p)+o_{C^2(U)}(1)\}.
 \label{eq:a-uniform-leading}
\end{equation}
Here $A_1(p)$ is smooth, positive, and invariant under $p\mapsto T(p)$. At $p=\pc$,
\begin{align}
 a_n(\pc)
 &=A_{1,\delta}n^{-1/2}e^{-n\Jn}\{1+o(1)\},
 \notag\\
 a_n'(\pc)&=0,
 \label{eq:a-center-first}\\
 a_n''(\pc)
 &=-I_1''(\pc)A_{1,\delta}n^{1/2}e^{-n\Jn}\{1+o(1)\},
 \label{eq:a-center-second}
\end{align}
where
\begin{equation}
 A_{1,\delta}:=A_1(\pc)=\frac1{2\delta\sqrt\pi}.
 \label{eq:A1}
\end{equation}
\end{lemma}

We next compute the wedge rate and its central derivative. For a high-centered wedge, the corresponding rate is
\begin{equation}
 I_H(p)
 =-3\log\left(q^{1/3}p^{2/3}
 +(1-q)^{1/3}(1-p)^{2/3}\right).
 \label{eq:IH}
\end{equation}
It is the constrained relative-entropy rate of forcing two low empirical proportions to cross a common high one. The $2{:}1$ weighting pulls the entropy-minimizing crossing toward the low mean. Complementation reverses this imbalance, giving the two wedge rates opposite, nonzero slopes at $\pc$.

\begin{lemma}[Wedge expansion]\label{lem:wedge}
At $p=\pc$, the wedge rate and its derivative are given by
\begin{gather}
 \Kn:=I_H(\pc)
 =-3\log\left(\frac{\taud}{1-\taud+\taud^2}\right),
 \label{eq:K}\\
 \kappa_\delta:=I_H'(\pc)
 =\frac{(1-\taud)^3(1+\taud)(1-\taud+\taud^2)}{\taud^3}>0.
 \label{eq:kappa}
\end{gather}
On some neighborhood $U$ of $\pc$, the high-centered wedge probability has the strong expansion
\begin{equation}
 H_n(p)=n^{-1}e^{-nI_H(p)}\{A_2(p)+o_{C^2(U)}(1)\}.
 \label{eq:H-uniform}
\end{equation}
Here $A_2(p)$ is smooth and positive. In particular,
\begin{align}
 H_n(\pc)
 &=A_{2,\delta}n^{-1}e^{-n\Kn}\{1+o(1)\},
 \notag\\
 H_n'(\pc)
 &=-\kappa_\delta A_{2,\delta}e^{-n\Kn}\{1+o(1)\},
 \label{eq:H-derivative}\\
 L_n(\pc)&=H_n(\pc),
 \qquad
 L_n'(\pc)=-H_n'(\pc),
 \label{eq:L-center}
\end{align}
where
\begin{equation*}
 A_{2,\delta}:=A_2(\pc)
 =\frac{1+\taud^2+\taud^4}
 {6\pi\sqrt3\,\taud(1-\taud)^2}.
\end{equation*}
Moreover, uniformly for $|p-\pc|\le n^{-1}$,
\begin{equation}
 |H_n''(p)|+|L_n''(p)|=O(ne^{-n\Kn}).
 \label{eq:wedge-second-bound}
\end{equation}
\end{lemma}

The lattice dimensions explain the polynomial orders needed below: the one-dimensional edge corner gives curvature of order $n^{1/2}e^{-n\Jn}$, whereas the two-dimensional wedge corner gives slope of order $e^{-n\Kn}$. The exponent gap is
\begin{equation*}
 \Gam=\Kn-\Jn.
\end{equation*}

\begin{lemma}[Exponent gap]\label{lem:gap}
The exponent gap has the equivalent forms
\begin{equation}
 \Gam
 =\log\frac{4(1-\taud+\taud^2)}{(1+\taud)^2}
 =\log\left[1+3\left(\frac{1-\taud}{1+\taud}\right)^2\right]>0.
 \label{eq:Gamma}
\end{equation}
Moreover,
\begin{equation*}
 \Jn<\Kn<2\Jn.
\end{equation*}
\end{lemma}

\begin{proof}
Using $1-\delta^2=4\taud^3/(1+\taud^3)^2$ and $1+\taud^3=(1+\taud)(1-\taud+\taud^2)$ in \eqref{eq:J-and-curvature} and \eqref{eq:K} gives the first formula in \eqref{eq:Gamma}. The second follows from
\[
 4(1-\taud+\taud^2)=(1+\taud)^2+3(1-\taud)^2,
\]
which also gives $\Kn>\Jn$. The inequality $2\Jn>\Kn$ is equivalent to
\[
 (1+\taud)^4(1-\taud+\taud^2)>16\taud^3.
\]
The difference between the two sides factors as
\[
 (1-\taud)^2(\taud^4+5\taud^3+12\taud^2+5\taud+1)>0.
\]
\end{proof}

The inequality $\Jn<\Kn$ makes the displacement exponentially smaller than the edge scale. The inequality $\Kn<2\Jn$ places the dependent wedge ahead of the disjoint-pair term $a_n^2$, so the wedge is the second exponential scale. \Cref{app:motif-calculations} proves the edge and wedge expansions.

\subsection{Rate separation beyond wedges}\label{sec:gap}

To control the remaining graph terms through two derivatives, we establish a strict uniform rate gap. For a nonempty edge set $A\subset E(K_{k-t,t})$, let $G_A$ be the bipartite graph formed by its incident vertices and edges. Its rate $\mathcal I_p(G_A)$ is the minimum total Bernoulli relative entropy over empirical proportions satisfying every comparison encoded by $G_A$. The precise variational definition and continuity argument are given in \Cref{app:graph-rate-details}; $I_1$ and $I_H$ are its one-edge and high-centered-wedge cases.

\begin{lemma}[Graph-rate separation]\label{lem:graph-gap}
\emph{Center classification and wedge-scale separation.} At $p=\pc$:
\begin{enumerate}[label=(\alph*)]
\item a one-edge component has rate $\Jn$;
\item a two-edge connected component has rate $\Kn$;
\item a matching of two edges has rate $2\Jn>\Kn$;
\item every connected component containing at least three edges has rate strictly larger than $\Kn$.
\end{enumerate}
Consequently, for fixed $k,t,\delta$, there is a neighborhood $U$ of $\pc$ and a number $\eta>0$ such that every inclusion--exclusion term in $\mathcal R_n$ has probability at most
\begin{equation}
 C(n+1)^ke^{-n(\Kn+2\eta)},
 \qquad p\in U.
 \label{eq:remainder-prob-bound}
\end{equation}
Its first two derivatives with respect to $p$ are $O(e^{-n(\Kn+\eta)})$, uniformly on $U$. Therefore
\begin{equation}
 \|\mathcal R_n\|_{C^2(U)}=O(e^{-n(\Kn+\eta)}).
 \label{eq:R-C2}
\end{equation}

\emph{Local edge-dominance separation.} After shrinking $U$ if necessary, there is also $\gamma>0$ such that every connected two-edge graph and every other non-single-edge graph occurring in \eqref{eq:exact-graph-decomp} satisfies
\begin{equation*}
 \mathcal I_p(G)\ge I_1(p)+\gamma,
 \qquad p\in U.
\end{equation*}
For $j=0,1,2$, the $j$th derivative of the sum of all such terms is bounded by
\begin{equation*}
 C_j n^{C_j}e^{-n(I_1(p)+\gamma)},
 \qquad p\in U.
\end{equation*}
\end{lemma}

\Cref{app:graph-rate-details} gives the proof. Here and below, $U$ denotes a fixed closed interval centered at $\pc$ and contained in $(0,1-\delta)$, chosen small enough that the local expansions and rate gaps in \Cref{lem:one-edge-rate,lem:wedge,lem:graph-gap} hold on the same interval. The first part isolates the wedge scale in the central derivative calculation below. The second part supplies the derivative control used to localize the maximizer in \Cref{sec:proof}; differentiating the finite likelihood sums adds only polynomial factors, which the exponential gap absorbs.

Differentiating \eqref{eq:exact-graph-decomp} at $\pc$ gives the two quantities in the optimizer quotient. Exact one-edge symmetry gives $a_n'(\pc)=0$ and hence $(a_n^2)'(\pc)=0$. By \Cref{lem:wedge,lem:graph-gap} and \eqref{eq:wedge-count-difference},
\begin{align}
 Q_n'(\pc)
 &=-N_HH_n'(\pc)-N_LL_n'(\pc)+o(e^{-n\Kn})
 \nonumber\\
 &=\frac{t(k-t)(k-2t)}2\kappa_\delta A_{2,\delta}e^{-n\Kn}\{1+o(1)\}.
 \label{eq:Q-prime-center}
\end{align}
The factor $k-2t$ is the multiplicity difference in \eqref{eq:wedge-count-difference} after removing the common factor $t(k-t)/2$. Since $a_n'(\pc)=0$,
\[
 (a_n^2)''(\pc)=2a_n(\pc)a_n''(\pc)=O(e^{-2n\Jn}).
\]
Together with \eqref{eq:wedge-second-bound} and \eqref{eq:R-C2}, the single-edge expansion gives
\[
 Q_n''(\pc)
 =t(k-t)a_n''(\pc)+O(ne^{-n\Kn})
  +O(e^{-2n\Jn})+O(e^{-n(\Kn+\eta)}).
\]
Since $0<\Jn<\Kn$, all three error terms are $o(n^{1/2}e^{-n\Jn})$. Therefore
\begin{equation}
 Q_n''(\pc)
 =-t(k-t)I_1''(\pc)A_{1,\delta}n^{1/2}e^{-n\Jn}\{1+o(1)\}.
 \label{eq:Q-second-center}
\end{equation}
Thus the edge term determines the denominator in \eqref{eq:intro-quotient}, while the oriented-wedge imbalance determines its numerator. The next section localizes the global optimizer to the central region and then applies this quotient.

\section{Proof of the main theorem}\label{sec:proof}

The derivative asymptotics at the end of \Cref{sec:edge-wedge} determine a stationary point near $\pc$. To use them for the global least-favorable location, we first localize the optimizer and prove uniqueness. The window $|p-\pc|\le n^{-1}$ is an intermediate scale: it is much wider than the final exponentially small displacement, but narrow enough for the edge curvature to be uniform. At its boundary, $n|I_1'(p)|$ is already of constant order, so the edge derivative has the required sign uniformly on each side and dominates the higher-rate graph terms outside the window.

\begin{proof}[Proof of \Cref{thm:global-top-t}]
The uniform estimates used below are established in \Cref{app:localization-details}. The one-edge rate $I_1$ has its unique minimum at $\pc$. The away-from-center upper bound \eqref{eq:Q-away-upper} and the center lower bound \eqref{eq:Q-center-lower} therefore put every global maximizer in a fixed neighborhood $U$ of $\pc$ for all large $n$.

On $U$, write $Q_n=t(k-t)a_n+B_n$, where $B_n$ is the sum of all non-single-edge terms in \eqref{eq:exact-graph-decomp}. The uniform single-edge expansion in \eqref{eq:a-uniform-leading} gives
\[
 \operatorname{sgn}a_n'(p)=-\operatorname{sgn}(p-\pc),
 \qquad n^{-1}\le |p-\pc|\le\operatorname{diam}(U),
\]
and \Cref{lem:graph-gap} makes $B_n'$ uniformly smaller than $a_n'$ on this region. This gives the outer derivative signs in \eqref{eq:Q-prime-outer-sign}.

On $|p-\pc|\le n^{-1}$, two differentiations of the single-edge expansion, together with the bounds for the non-single-edge terms, give
\begin{equation}
 Q_n''(p)
 =-t(k-t)I_1''(\pc)A_{1,\delta}n^{1/2}e^{-n\Jn}\{1+o(1)\}
 \label{eq:uniform-central-curvature}
\end{equation}
uniformly: the wedge bound \eqref{eq:wedge-second-bound}, the disjoint-pair estimate $(a_n^2)''=O(e^{-2n\Jn})$, and the remainder bound \eqref{eq:R-C2} show that $B_n''$ is of smaller order. Hence $Q_n$ is strictly concave on the central interval. Together with the outer derivative signs, this proves that $Q_n$ has a unique global maximizer $p_{n,k,t}^*$, which lies within $n^{-1}$ of $\pc$.

If $k=2t$, exact reflection makes $Q_n$ symmetric about $\pc$; eventual uniqueness therefore forces its maximizer to be $p_{n,k,t}^*=\pc$, and \eqref{eq:beyond-all-orders} is immediate.

Assume $k\ne2t$. For all sufficiently large $n$, the central interval lies in $(0,1-\delta)$, so $p_{n,k,t}^*$ is interior and $Q_n'(p_{n,k,t}^*)=0$. The mean-value theorem gives a point $\xi_n$ between $\pc$ and $p_{n,k,t}^*$ such that
\begin{equation}
 p_{n,k,t}^*-\pc=-\frac{Q_n'(\pc)}{Q_n''(\xi_n)}.
 \label{eq:MVT-shift}
\end{equation}
Equations \eqref{eq:Q-prime-center}, \eqref{eq:uniform-central-curvature}, and \eqref{eq:MVT-shift} first give $|p_{n,k,t}^*-\pc|=O(n^{-1/2}e^{-n\Gam})$. The uniform curvature estimate then gives $Q_n''(\xi_n)/Q_n''(\pc)\to1$. Substituting \eqref{eq:Q-prime-center} and \eqref{eq:Q-second-center} into \eqref{eq:MVT-shift} yields
\[
 p_{n,k,t}^*-\pc
 =(k-2t)\frac{\kappa_\delta A_{2,\delta}}
 {2I_1''(\pc)A_{1,\delta}}
 n^{-1/2}e^{-n\Gam}\{1+o(1)\}.
\]
Since $\delta=(1-\taud^3)/(1+\taud^3)$ and $1+\taud^3=(1+\taud)(1-\taud+\taud^2)$, direct substitution gives
\[
 \frac{\kappa_\delta A_{2,\delta}}
 {2I_1''(\pc)A_{1,\delta}}
 =\frac{\taud^2}
 {3\sqrt{3\pi}(1+\taud)^2(1-\taud+\taud^2)}
 =\Cd.
\]
This proves \eqref{eq:select-t-main}. Equation \eqref{eq:beyond-all-orders} follows from $\Gam>0$.
\end{proof}

The location expansion also quantifies the smaller-order error from evaluating the PCS at $\pc$ rather than at the exact least-favorable location. To leading order, this loss is one half of the PCS curvature at the center times $(p_{n,k,t}^*-\pc)^2$.

\begin{corollary}[Worst PCS and the cost of using the symmetric center]\label{cor:pcs-loss}
Fix $k\ge2$, $1\le t<k$, and $\delta\in(0,1)$. Then, as $n\to\infty$,
\begin{equation}
 1-P_{n,k,t,\delta}(p_{n,k,t}^*)
 =t(k-t)A_{1,\delta}n^{-1/2}e^{-n\Jn}\{1+o(1)\}.
 \label{eq:worst-pcs-leading}
\end{equation}
If $k\ne2t$, then
\begin{align*}
 &P_{n,k,t,\delta}(\pc)
   -P_{n,k,t,\delta}(p_{n,k,t}^*)
 \\[-1mm]
 &\qquad=\frac{t(k-t)(k-2t)^2\kappa_\delta^2A_{2,\delta}^2}
 {8I_1''(\pc)A_{1,\delta}}
 n^{-1/2}e^{-n(2\Kn-\Jn)}\{1+o(1)\}.
\end{align*}
Thus the difference between the centered PCS and the worst-case PCS is on the smaller exponential scale $2\Kn-\Jn$.
\end{corollary}

\begin{proof}
By \Cref{thm:global-top-t}, $|p_{n,k,t}^*-\pc|=O(n^{-1/2}e^{-n\Gam})$; in the balanced case this difference is eventually zero. Since $I_1'(\pc)=0$ and $I_1(\pc)=\Jn$,
\[
 n\{I_1(p_{n,k,t}^*)-\Jn\}
 =O\!\left(n|p_{n,k,t}^*-\pc|^2\right)
 =O(e^{-2n\Gam})=o(1).
\]
The uniform single-edge expansion and the local graph-rate gap therefore give
\[
 Q_n(p_{n,k,t}^*)
 =t(k-t)a_n(p_{n,k,t}^*)\{1+o(1)\}
 =t(k-t)A_{1,\delta}n^{-1/2}e^{-n\Jn}\{1+o(1)\},
\]
which proves \eqref{eq:worst-pcs-leading}.

Suppose $k\ne2t$ and put $h_n=p_{n,k,t}^*-\pc$. The point $p_{n,k,t}^*$ is interior for all large $n$, so $Q_n'(p_{n,k,t}^*)=0$. The mean-value theorem and \eqref{eq:uniform-central-curvature} yield
\[
 h_n
 =-\frac{Q_n'(\pc)}{Q_n''(\pc)}\{1+o(1)\}.
\]
For some $\zeta_n$ between $\pc$ and $p_{n,k,t}^*$,
\[
 Q_n(p_{n,k,t}^*)-Q_n(\pc)
 =Q_n'(\pc)h_n+\frac12Q_n''(\zeta_n)h_n^2.
\]
The uniform curvature estimate gives $Q_n''(\zeta_n)/Q_n''(\pc)\to1$. Substituting the preceding expression for $h_n$ therefore gives
\[
 Q_n(p_{n,k,t}^*)-Q_n(\pc)
 =-\frac{Q_n'(\pc)^2}{2Q_n''(\pc)}\{1+o(1)\}.
\]
Substituting \eqref{eq:Q-prime-center} and \eqref{eq:Q-second-center} gives the second assertion.
\end{proof}

\section{Numerical illustration}\label{sec:numerical}

We compare finite-sample minimizing locations with their asymptotic approximations at $\delta=0.4$. The finite binomial-sum expression for $P_{n,k,t,\delta}$ and its analytic derivative are evaluated in 80-decimal arithmetic. We bracket derivative sign changes, refine them by bisection, and compare the resulting locations with the endpoints. For $k=4$, write the leading approximation from \Cref{thm:global-top-t} as
\[
 \widetilde p_{n,4,t}
 =\pc+(4-2t)\Cd n^{-1/2}e^{-n\Gam}.
\]
\Cref{fig:numerical-illustration}(a) compares the minimizing locations with $\widetilde p_{n,4,t}$ for $t=1,2,3$. The $t=1$ and $t=3$ curves are exact reflections about $\pc=0.3$; at $n=1$, their minimizers are the right and left endpoints, respectively. When $n=1$ and $t=2$, both endpoints minimize; the computed location is at the center in the remaining displayed balanced cases. The three curves then draw rapidly toward $\pc$. \Cref{fig:numerical-illustration}(b) places the $t=1$ displacement on a logarithmic axis and extends the comparison to $n=100$.

\begin{figure}[!htbp]
\centering
\includegraphics[width=\textwidth]{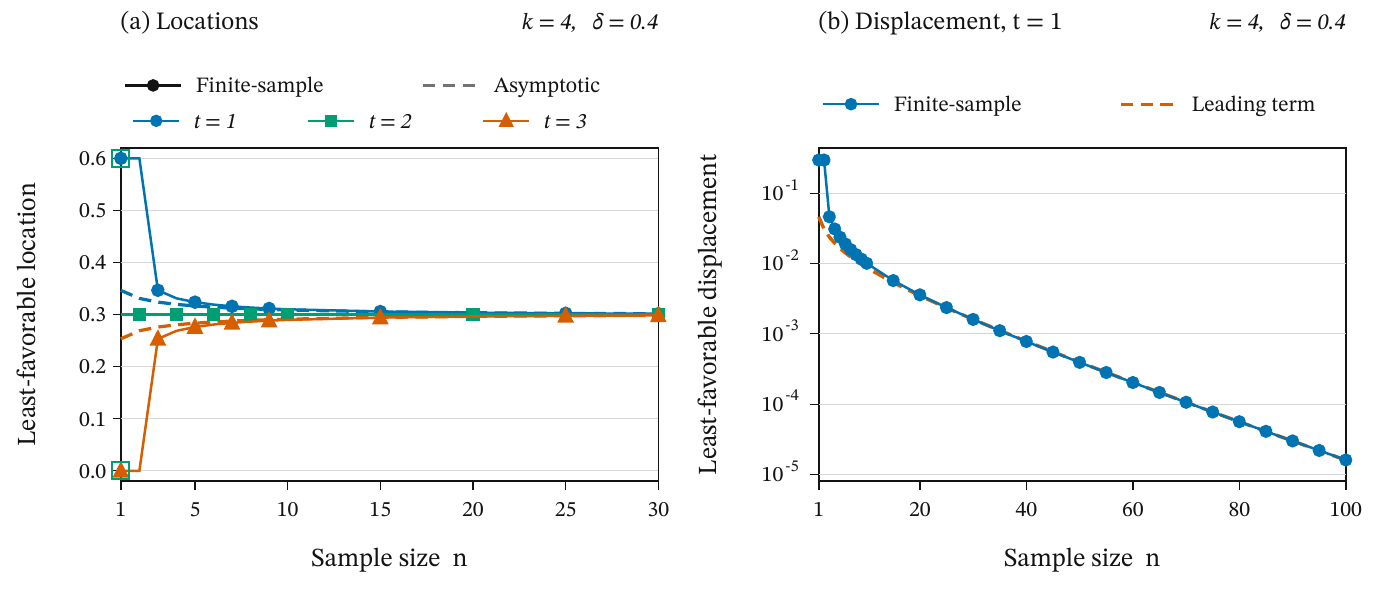}
\caption{Least-favorable locations and their leading approximations for
$k=4$ and $\delta=0.4$. (a) Locations for $t=1,2,3$ and
$n=1,\ldots,10,15,20,25,30$; the open squares at $n=1$ mark the two
minimizers for $t=2$. (b) Finite-sample displacement $p_{n,4,1}^*-\pc$ and the
leading term $2\Cd n^{-1/2}e^{-n\Gam}$ for
$n=1,\ldots,10,15,20,\ldots,100$; the vertical axis is logarithmic.}
\label{fig:numerical-illustration}
\end{figure}

\FloatBarrier

For comparisons across $(k,t)$ configurations, define the scale-free ratio
\begin{equation*}
 \rho_n(k,t,\delta)
 =\frac{p_{n,k,t}^*-\pc}
 {(k-2t)\Cd n^{-1/2}e^{-n\Gam}},
 \qquad k\ne2t.
\end{equation*}
By \Cref{thm:global-top-t}, $\rho_n(k,t,\delta)\to1$.

\Cref{tab:scaled-ratios} compares three configurations at $\delta=0.4$ over $n=20,40,80$. The $(6,4)$ and $(6,2)$ columns display the exact reflection, while $(6,2)$ and $(4,1)$ have the same value $k-2t=2$ and show the remaining finite-sample dependence on $(k,t)$.

\begin{table}[ht]
\centering
\small
\caption{Displacements and scaled ratios at $\delta=0.4$. The center is
$\pc=0.3$.}
\label{tab:scaled-ratios}
\begin{tabular}{c@{\quad}cc@{\quad}cc@{\quad}cc}
\toprule
& \multicolumn{2}{c}{$(k,t)=(6,4)$}
& \multicolumn{2}{c}{$(k,t)=(6,2)$}
& \multicolumn{2}{c}{$(k,t)=(4,1)$}\\
\cmidrule(lr){2-3}\cmidrule(lr){4-5}\cmidrule(lr){6-7}
$n$ & $p_{n,k,t}^*-\pc$ & $\rho_n$
    & $p_{n,k,t}^*-\pc$ & $\rho_n$
    & $p_{n,k,t}^*-\pc$ & $\rho_n$\\
\midrule
20 & $-2.939886\times10^{-3}$ & 0.840530
   & $ 2.939886\times10^{-3}$ & 0.840530
   & $ 3.598146\times10^{-3}$ & 1.028730\\
40 & $-6.977117\times10^{-4}$ & 0.888502
   & $ 6.977117\times10^{-4}$ & 0.888502
   & $ 7.787118\times10^{-4}$ & 0.991651\\
80 & $-5.424220\times10^{-5}$ & 0.969000
   & $ 5.424220\times10^{-5}$ & 0.969000
   & $ 5.602576\times10^{-5}$ & 1.000862\\
\bottomrule
\end{tabular}
\end{table}

The balanced calculations reveal a separate finite-sample pattern. \Cref{tab:balanced-locations} records whether the numerical search returns only the center as a minimizing location when $k=2t$. This occurs at every displayed even sample size, whereas off-center minima appear successively across the odd sample sizes as $t$ increases.

\begin{table}[ht]
\centering
\small
\caption{Finite-sample centering in balanced cases at $\delta=0.4$. Here
$k=2t$ and $\pc=0.3$. Each entry applies to every integer $t$ in the
indicated row. A checkmark means that the search returns $\pc$ as its only
minimizing location. At
$n=1$, the objective is flat for $t=1$, while for $t\ge2$ its two endpoints
minimize. Every other dash denotes a reflected pair of off-center minima.}
\label{tab:balanced-locations}
\setlength{\tabcolsep}{7pt}
\begin{tabular}{c@{\qquad}*{10}{c}}
\toprule
& \multicolumn{10}{c}{sample size $n$}\\
\cmidrule(lr){2-11}
$t$ & 1 & 2 & 3 & 4 & 5 & 6 & 7 & 8 & 9 & 10\\
\midrule
$1\le t\le6$ & -- & $\checkmark$ & $\checkmark$ & $\checkmark$ & $\checkmark$ & $\checkmark$ & $\checkmark$ & $\checkmark$ & $\checkmark$ & $\checkmark$\\
$7\le t\le10$ & -- & $\checkmark$ & -- & $\checkmark$ & $\checkmark$ & $\checkmark$ & $\checkmark$ & $\checkmark$ & $\checkmark$ & $\checkmark$\\
$11\le t\le15$ & -- & $\checkmark$ & -- & $\checkmark$ & -- & $\checkmark$ & $\checkmark$ & $\checkmark$ & $\checkmark$ & $\checkmark$\\
$16\le t\le20$ & -- & $\checkmark$ & -- & $\checkmark$ & -- & $\checkmark$ & -- & $\checkmark$ & $\checkmark$ & $\checkmark$\\
\bottomrule
\end{tabular}
\end{table}

\FloatBarrier

\section{Conclusion}\label{sec:conclusion}

The symmetric center $\pc=(1-\delta)/2$ is the natural candidate for the least-favorable location, as suggested by the Gaussian approximation. For general binomial top-$t$ selection, \Cref{thm:global-top-t} shows that, in the balanced case $k=2t$, this center is the exact least-favorable location for all sufficiently large but finite $n$; in the unbalanced case $k\ne2t$, the exact location differs from it by an explicitly determined, exponentially small amount whose direction is given by the sign of $k-2t$. Even in the latter case, \Cref{cor:pcs-loss} shows that the PCS at $\pc$ differs from the worst-case PCS on the still smaller exponential scale $2\Kn-\Jn$.

The comparison graph explains why the exact location lies beyond the leading PCS approximation. The edge term gives the leading PCS error and central curvature, but exact symmetry makes its central slope zero. The first nonzero slope comes from the multiplicity imbalance between the two wedge orientations. Thus the center remains correct to every algebraic order, even though the exact location is off center in unbalanced problems.

In fact, as \Cref{tab:balanced-locations} suggests, we further conjecture that, in every balanced problem, the symmetric center is the unique least-favorable location for every even sample size $n\ge2$. We leave this conjecture for future work.

\appendix

\section{Finite-sample reduction and single-edge centering}\label[appendix]{app:finite-sample}

This appendix proves the full-space reduction and the exact centering of a single low--high comparison.

\subsection{Proof of the full-space reduction}\label[appendix]{app:full-reduction-proof}

\begin{proof}[Proof of \Cref{prop:full-reduction}]
Let $L$ and $H$ be the index sets of the lower $k-t$ and upper $t$ populations, respectively, and put $p=\theta_{[k-t+1]}-\delta$ and $q=p+\delta$. The gap condition gives $p\in[0,1-\delta]$. From $\theta_{[k-t+1]}-\theta_{[k-t]}\ge\delta$ and the definition of $p$,
\[
 \theta_i\le p\quad(i\in L),
 \qquad
 \theta_j\ge q\quad(j\in H).
\]
For $i=1,\ldots,k$ and $m=1,\ldots,n$, take independent $W_{im}\sim\mathrm{Unif}(0,1)$ and construct the monotone coupling
\[
 B_{im}(a)=\one\{W_{im}\le a\},
 \qquad
 S_i(a)=\sum_{m=1}^nB_{im}(a).
\]
Then $a\le b$ implies $S_i(a)\le S_i(b)$ for every realization. Using the same tie-breaking variables $U_i$, define the original scores by
\[
 \widetilde S_r^{\boldsymbol\theta}=(S_r(\theta_r),U_r)
 \qquad(1\le r\le k)
\]
and the reduced scores by
\[
 \widetilde S_i^p=(S_i(p),U_i)\quad(i\in L),
 \qquad
 \widetilde S_j^p=(S_j(q),U_j)\quad(j\in H).
\]
For every realization,
\[
 \widetilde S_i^{\boldsymbol\theta}\le\widetilde S_i^p\quad(i\in L),
 \qquad
 \widetilde S_j^{\boldsymbol\theta}\ge\widetilde S_j^p\quad(j\in H).
\]
Let
\begin{equation*}
 C_{\boldsymbol\theta}=\left\{\max_{i\in L}\widetilde S_i^{\boldsymbol\theta}<\min_{j\in H}\widetilde S_j^{\boldsymbol\theta}\right\},
 \qquad
 C_p=\left\{\max_{i\in L}\widetilde S_i^p<\min_{j\in H}\widetilde S_j^p\right\}.
\end{equation*}
Under this coupling, the preceding inequalities give
\[
 C_p\subseteq C_{\boldsymbol\theta},
\]
and hence \eqref{eq:pointwise-reduction}.

For strictness, let $p\in(0,1-\delta)$, so $0<p<q<1$. If $i\in L$ satisfies $\theta_i<p$, let $A_i^-$ be the event on which
\[
 W_{im}\in(\theta_i,p]\quad(1\le m\le n),
 \qquad
 W_{\ell m}>p\quad(\ell\in L\setminus\{i\},\ 1\le m\le n),
\]
\[
 W_{jm}\le q\quad(j\in H,\ 1\le m\le n),
 \qquad
 U_i>\min_{j\in H}U_j.
\]
Independence gives
\[
 \Pp(A_i^-)
 =(p-\theta_i)^n(1-p)^{n(k-t-1)}q^{nt}\frac{t}{t+1}>0.
\]
On $A_i^-$, all high counts are $n$ in both configurations, whereas the $i$th low count rises from $0$ to $n$. The auxiliary ordering makes the original selection correct and the reduced selection incorrect. Hence $A_i^-\subseteq C_{\boldsymbol\theta}\setminus C_p$.

If instead $j\in H$ satisfies $\theta_j>q$, let $A_j^+$ be the event on which
\[
 W_{jm}\in(q,\theta_j]\quad(1\le m\le n),
 \qquad
 W_{im}>p\quad(i\in L,\ 1\le m\le n),
\]
\[
 W_{\ell m}\le q\quad(\ell\in H\setminus\{j\},\ 1\le m\le n),
 \qquad
 \max_{i\in L}U_i>U_j.
\]
Then
\[
 \Pp(A_j^+)
 =(\theta_j-q)^n(1-p)^{n(k-t)}q^{n(t-1)}
   \frac{k-t}{k-t+1}>0.
\]
On $A_j^+$, the $j$th high count falls from $n$ to $0$, while every low count remains $0$ and every other high count remains $n$. The auxiliary ordering again makes the original selection correct and the reduced one incorrect, so $A_j^+\subseteq C_{\boldsymbol\theta}\setminus C_p$. Thus, in either case,
\[
 \Pp(C_{\boldsymbol\theta}\setminus C_p)>0,
\]
which makes \eqref{eq:pointwise-reduction} strict.

If neither strict coordinate inequality occurs, then every low parameter is $p$ and every high parameter is $q$; the two configurations coincide. This proves the stated equality characterization.

Finally, every two-level configuration with $k-t$ entries $p$ and $t$ entries $p+\delta$ belongs to $\ThetaPZ$. Taking infima in \eqref{eq:pointwise-reduction} gives one inequality in \eqref{eq:full-reduction}, and this inclusion gives the reverse inequality. Since the scalar objective is continuous, the infimum on the right is a minimum.
\end{proof}

\subsection{Exact centering of a single edge}\label[appendix]{app:one-edge-centering}

For $k=2$ and $t=1$, incorrect selection is exactly one low--high misranking, so this calculation proves the scalar assertion in \Cref{cor:two-population-location}. Pair one observation $B_p\sim\Bern(p)$ with an independent observation $B_{p+\delta}\sim\Bern(p+\delta)$, and put $Z=B_p-B_{p+\delta}$. With
\[
 \alpha=p(1-p-\delta)=\pc^2-(p-\pc)^2,
\]
its three masses are
\[
 \Pp(Z=1)=\alpha,\qquad
 \Pp(Z=-1)=\alpha+\delta,\qquad
 \Pp(Z=0)=1-2\alpha-\delta.
\]
Let $c_{m,s}(\alpha)$ be the probability that the sum of $m$ independent copies of $Z$ equals $s$. The Laurent generating function is
\[
 \sum_s c_{m,s}(\alpha)z^s
 =\{\alpha z+(1-2\alpha-\delta)+(\alpha+\delta)z^{-1}\}^m.
\]
Writing the tie-weighted upper-tail probability as
\[
 a_n(p)=\sum_s
 \left\{\one\{s\ge1\}+\frac12\one\{s=0\}\right\}
 c_{n,s}(\alpha),
\]
differentiation of the generating function gives
\[
 \frac{d a_n}{d\alpha}
 =\frac n2\{c_{n-1,-1}(\alpha)-c_{n-1,1}(\alpha)\}.
\]
Interchanging the $+1$ and $-1$ steps in every path with net displacement $1$ yields
\[
 c_{m,-1}(\alpha)
 =\frac{\alpha+\delta}{\alpha}c_{m,1}(\alpha).
\]
Hence, for $0<\alpha\le\pc^2$,
\[
 \frac{d a_n}{d\alpha}
 =\frac{n\delta}{2\alpha}c_{n-1,1}(\alpha)>0
 \qquad(n\ge2).
\]
Continuity covers $\alpha=0$. Since $\alpha$ is uniquely maximized at $p=\pc$, so is $a_n(p)$. When $n=1$, $c_{0,1}=0$ and $a_1$ is constant.

\section{Weighted lattice asymptotics and motif calculations}\label[appendix]{app:motif-calculations}

This appendix supplies the two expansions used in \Cref{sec:edge-wedge}: the single-edge curvature and the wedge slope. For each motif, a KL minimization identifies the constrained empirical configuration and its rate, the equivalent saddlepoint constructs the exponential tilt, and the lattice-corner expansion gives the prefactor and parameter derivatives. The edge uses the stationary-rate case with $d=1$; the wedge uses the nonstationary-rate case with $d=2$.

The exponential-tilting and local-limit argument belongs to the strong large-deviation theory developed by Bahadur and Ranga Rao \cite{BahadurRao1960}. Chaganty and Sethuraman connect local limit theorems to strong large-deviation asymptotics for scalar and vector sequences, respectively, in \cite{ChagantySethuraman1993,ChagantySethuraman1996}. Barbe and Broniatowski \cite{BarbeBroniatowski2005} give a multidimensional sharp large-deviation treatment for sums of random vectors, including the lattice case. The form needed here is parameter-uniform, lattice-valued, and valid through two derivatives; the tie-breaking weights enter only the prefactor.

For $a\in(0,1)$ and $x\in[0,1]$, write
\begin{equation}
 \KL(x\|a)=x\log\frac xa+(1-x)\log\frac{1-x}{1-a},
 \label{eq:KL}
\end{equation}
with the standard endpoint convention.

Throughout this appendix, $q=p+\delta$ and
\[
 \ell:=\pc,\qquad
 h:=1-\pc=\frac{1+\delta}{2},\qquad
 \taud=\left(\frac{1-\delta}{1+\delta}\right)^{1/3}.
\]

For a compact interval $U\subset\Rr$ and an integer $r\ge0$, write $f_n=O_{C^r(U)}(\varepsilon_n)$ when $f_n$ is $r$ times continuously differentiable and
\[
 \max_{0\le j\le r}\sup_{p\in U}|f_n^{(j)}(p)|
 \le C\varepsilon_n.
\]
The notation $o_{C^r(U)}(1)$ has the analogous meaning. The applications below use $d\in\{1,2\}$, $r=2$, and the nonnegative weights specified in the next two subsections; their corresponding weighted geometric sums $W$ are positive.

\begin{lemma}[Parametric weighted lattice-corner expansion]\label{lem:corner}
Let $U$ be a compact interval. For each $p\in U$, let $Z_p$ be a random vector supported by a fixed finite subset of $\Zz^d$, with masses that are positive and $C^{r+2}$ in $p$. Assume that the additive group generated by differences of support points is $\Zz^d$. Put
\[
 M_p(\lambda)=\Ee e^{\langle\lambda,Z_p\rangle}.
\]
Suppose there is a $C^{r+2}$ map $\lambda:U\to(0,\infty)^d$ satisfying
\[
 \nabla_\lambda\log M_p(\lambda(p))=0.
\]
Let
\[
 I(p)=-\log M_p(\lambda(p)),
 \qquad
 \Sigma(p)=\nabla_\lambda^2\log M_p(\lambda(p)).
\]
Let $w:\Zz_+^d\to\Rr$ be bounded and define
\[
 W(p)=\sum_{z\in\Zz_+^d}
 w(z)e^{-\langle\lambda(p),z\rangle}.
\]
For independent copies $Z_{p,1},\ldots,Z_{p,n}$, put $S_{n,p}=\sum_{i=1}^nZ_{p,i}$ and
\[
 \mathcal T_n(p)=\sum_{z\in\Zz_+^d}w(z)\Pp(S_{n,p}=z).
\]
Then, for $\varepsilon_n=n^{-1/2}(\log n)^{3r+6}$,
\begin{equation}
 \mathcal T_n(p)
 =\frac{e^{-nI(p)}n^{-d/2}}
 {(2\pi)^{d/2}\sqrt{\det\Sigma(p)}}
 \left[W(p)+O_{C^r(U)}(\varepsilon_n)\right].
 \label{eq:corner-expansion}
\end{equation}
If $r\ge1$ and $I'(p_0)\ne0$, then
\begin{equation}
 \mathcal T_n'(p_0)
 =n^{1-d/2}e^{-nI(p_0)}
 \left\{-I'(p_0)
 \frac{W(p_0)}{(2\pi)^{d/2}\sqrt{\det\Sigma(p_0)}}+o(1)\right\}.
 \label{eq:corner-first-derivative}
\end{equation}
If in addition $W(p_0)\ne0$, this may equivalently be written with a relative factor $\{1+o(1)\}$. If $r\ge2$, $I'(p_0)=0$, $I''(p_0)>0$, and $W(p_0)>0$, then
\begin{equation}
 \mathcal T_n''(p_0)
 =-I''(p_0)
 \frac{W(p_0)}{(2\pi)^{d/2}\sqrt{\det\Sigma(p_0)}}
 n^{1-d/2}e^{-nI(p_0)}\{1+o(1)\}.
 \label{eq:corner-second-derivative}
\end{equation}
\end{lemma}

The proof is deferred to \Cref{app:corner-proof}. We first verify the lemma's hypotheses and compute the edge and wedge constants.

\subsection{Single-edge rate and prefactor}\label[appendix]{app:single-edge-asymptotics}

Let $X_0\sim\Bern(p)$ and $Y_0\sim\Bern(q)$ be independent, and put $Z=X_0-Y_0$. Its moment generating function is
\begin{equation*}
 M_{1,p}(\lambda)
 =(1-p+pe^\lambda)(1-q+qe^{-\lambda}).
\end{equation*}
The saddlepoint solving $\partial_\lambda\log M_{1,p}(\lambda)=0$ is
\begin{equation}
 e^{\lambda_{\mathrm e}(p)}
 =\left[\frac{q(1-p)}{p(1-q)}\right]^{1/2},
 \label{eq:lambda-edge}
\end{equation}
and the associated rate $-\log M_{1,p}(\lambda_{\mathrm e}(p))$ is the function $I_1(p)$ in \eqref{eq:I1}. Its KL representation is
\begin{equation*}
 I_1(p)
 =\inf_{x\ge y}\{\KL(x\|p)+\KL(y\|q)\}
 =\inf_{0\le z\le1}\{\KL(z\|p)+\KL(z\|q)\}.
\end{equation*}
To see the second equality, let $F(x,y)=\KL(x\|p)+\KL(y\|q)$. If a minimizer had $x>y$, then for all sufficiently small $s>0$ the point $(1-s)(x,y)+s(p,q)$ would remain feasible, while strict convexity and $F(p,q)=0$ would give a smaller value. Thus the constraint is active at the minimizer, so $x=y=z$.

\begin{proof}[Proof of \Cref{lem:one-edge-rate}]
Formula \eqref{eq:I1} is invariant under $p\mapsto T(p)$.

To prove uniqueness of the minimum, write $s=p+\delta/2$. The affinity inside the logarithm in \eqref{eq:I1} is
\[
 B(s)=\sqrt{s^2-\delta^2/4}
 +\sqrt{(1-s)^2-\delta^2/4}.
\]
Each square-root term is strictly concave on the interior and $B(s)=B(1-s)$. Thus $B$ has its unique maximum at $s=1/2$, equivalently $p=\pc$. Direct substitution and differentiation give \eqref{eq:J-and-curvature}.

Choose $r_0\in(0,\pc)$ and set $U=[\pc-r_0,\pc+r_0]$, so that $U\subset(0,1-\delta)$ and $T(U)=U$. On $U$, the support of $Z$ is the fixed set $\{-1,0,1\}$, all three masses are positive, and the support differences generate $\Zz$. Formula \eqref{eq:lambda-edge} defines a smooth positive saddlepoint because $q>p$. With
\[
 w_1(0)=\frac12,\qquad w_1(z)=1\quad(z\ge1),
\]
the weighted corner probability in \Cref{lem:corner} is exactly $a_n(p)$. Applying the lemma with $d=1$ and $r=2$ yields \eqref{eq:a-uniform-leading} for a smooth positive prefactor $A_1(p)$. Moreover,
\[
 A_1(p)=\lim_{n\to\infty}n^{1/2}e^{nI_1(p)}a_n(p)
\]
uniformly on the neighborhood. The invariance of $a_n$ from \Cref{prop:reflection}, together with the corresponding invariance of $I_1$, therefore gives $A_1(p)=A_1(T(p))$.

At the center, \eqref{eq:lambda-edge} becomes $e^{\lambda_{\mathrm e}}=h/\ell$. Under the tilted law, $X_0$ and $Y_0$ are independent $\Bern(1/2)$ variables, so the tilted difference has variance $1/2$ and full lattice span. Put $\rho_{\mathrm e}=e^{-\lambda_{\mathrm e}}=\ell/h=(1-\delta)/(1+\delta)$. The weighted geometric factor associated with $w_1$ is
\[
 W_1=\frac12+\sum_{z=1}^\infty\rho_{\mathrm e}^z
 =\frac1{2\delta}.
\]
Since $(2\pi)^{1/2}\sqrt{1/2}=\sqrt\pi$, the leading formula gives \eqref{eq:A1}. Exact reflection gives \eqref{eq:a-center-first}, and \eqref{eq:corner-second-derivative} gives \eqref{eq:a-center-second}.
\end{proof}

\subsection{Wedge rate and prefactor}\label[appendix]{app:wedge-asymptotics}

The following reduction identifies the constrained empirical configuration for a high-centered wedge.

\begin{lemma}[Variational reduction for a high-centered wedge]\label{lem:wedge-variational}
For $0<p<q<1$,
\begin{align*}
 &\inf\left\{
 \KL(x_1\|p)+\KL(x_2\|p)+\KL(y\|q):
 (x_1,x_2,y)\in[0,1]^3,\ x_1\ge y,\ x_2\ge y
 \right\}
 \\
 &\hspace{25mm}
 =\inf_{0\le z\le1}\{2\KL(z\|p)+\KL(z\|q)\}.
\end{align*}
The minimizer is unique and has $x_1=x_2=y=z^*(p)\in(p,q)$.
\end{lemma}

\begin{proof}
For fixed $y$, each low term is minimized over $x_i\ge y$ by $x_i=\max\{p,y\}$. If $y<p$, increasing $y$ toward $p$ while keeping $x_1=x_2=p$ strictly decreases $\KL(y\|q)$, so no minimizer has $y<p$. Consequently $y\ge p$ and both low constraints are active: $x_1=x_2=y$. The problem reduces to the strictly convex function
\[
 f(z)=2\KL(z\|p)+\KL(z\|q),\qquad z\in[p,1].
\]
Its right derivative at $p$ is $\partial_z\KL(p\|q)<0$, whereas its derivative at $q$ is $2\partial_z\KL(q\|p)>0$. Hence its unique minimizer lies in $(p,q)$, proving the claim.
\end{proof}

For one Bernoulli trial in a high-centered wedge, put
\[
 Z=(X_{01}-Y_0,\,X_{02}-Y_0)\in\Zz^2,
\]
where $X_{01},X_{02}\sim\Bern(p)$ and $Y_0\sim\Bern(q)$ are independent. Its moment generating function is
\begin{equation}
 M_{2,p}(\lambda_1,\lambda_2)
 =(1-p+pe^{\lambda_1})(1-p+pe^{\lambda_2})
 (1-q+qe^{-\lambda_1-\lambda_2}).
 \label{eq:M2}
\end{equation}
The symmetric saddlepoint has $\lambda_1=\lambda_2=\lambda_{\mathrm w}(p)$ with
\begin{equation}
 e^{\lambda_{\mathrm w}(p)}
 =\left[\frac{q(1-p)}{p(1-q)}\right]^{1/3}.
 \label{eq:lambda-wedge}
\end{equation}
By \Cref{lem:wedge-variational}, the corresponding rate is $I_H(p)$ in \eqref{eq:IH}. The unique minimizer satisfies
\[
 \left(\frac{z^*}{1-z^*}\right)^3
 =\frac{q}{1-q}\left(\frac p{1-p}\right)^2.
\]
The low-centered orientation has rate $I_H(T(p))$.

Choose a compact neighborhood $U\subset(0,1-\delta)$ of $\pc$. The support of $Z$ is the fixed set
\[
 \{(-1,-1),(0,-1),(-1,0),(0,0),(1,0),(0,1),(1,1)\}.
\]
Every support point has positive mass on $U$, and the support differences include the two coordinate vectors, so they generate $\Zz^2$. Formula \eqref{eq:lambda-wedge} defines a smooth saddlepoint $(\lambda_{\mathrm w}(p),\lambda_{\mathrm w}(p))\in(0,\infty)^2$. The corner weight is
\[
 w_2(a,b)=
 \begin{cases}
  1,&a,b\ge1,\\
  1/2,&\text{exactly one of $a,b$ is zero},\\
  1/3,&a=b=0.
 \end{cases}
\]
Thus the high-centered wedge probability is the weighted corner probability in \Cref{lem:corner}.

\begin{proof}[Proof of \Cref{lem:wedge}]
At the center, $\ell=\taud^3/(1+\taud^3)$ and $h=1/(1+\taud^3)$. Substitution in \eqref{eq:IH} gives \eqref{eq:K}, using $1+\taud^3=(1+\taud)(1-\taud+\taud^2)$. Differentiating \eqref{eq:IH} and substituting $p=\ell,q=h$ gives \eqref{eq:kappa}; positivity follows from $0<\taud<1$.

At the center, $e^{\lambda_{\mathrm w}}=\taud^{-2}$. Under the tilted law associated with \eqref{eq:M2}, the three Bernoulli variables become independent with the common success probability
\[
 \pi_\delta=\frac{\taud}{1+\taud}.
\]
Writing $v_\delta=\pi_\delta(1-\pi_\delta)=\taud/(1+\taud)^2$, the covariance of $(X_{01}-Y_0,X_{02}-Y_0)$ is
\begin{equation*}
 \Sigma_2=v_\delta
 \begin{pmatrix}2&1\\1&2\end{pmatrix},
 \qquad
 \sqrt{\det\Sigma_2}=\sqrt3\,v_\delta.
\end{equation*}

The auxiliary uniforms give the stated boundary weights. In particular, the weight at $(0,0)$ is
\[
 \Pp(U_1>V,U_2>V)=\int_0^1(1-v)^2\,dv=\frac13.
\]
Let $\rho_{\mathrm w}=e^{-\lambda_{\mathrm w}}=\taud^2$. The weighted geometric sum is
\begin{equation*}
 W_2
 =\sum_{a,b\ge1}\rho_{\mathrm w}^{a+b}
 +\frac12\sum_{a\ge1}\rho_{\mathrm w}^a
 +\frac12\sum_{b\ge1}\rho_{\mathrm w}^b+\frac13
 =\frac{1+\taud^2+\taud^4}{3(1-\taud^2)^2}.
\end{equation*}
Applying \Cref{lem:corner} with $d=2$ and $r=2$ gives the uniform expansion \eqref{eq:H-uniform} and
\[
 A_{2,\delta}=\frac{W_2}{2\pi\sqrt3\,v_\delta}
 =\frac{1+\taud^2+\taud^4}
 {6\pi\sqrt3\,\taud(1-\taud)^2}.
\]
Because $I_H'(\pc)=\kappa_\delta\ne0$ and $W_2>0$, \eqref{eq:corner-first-derivative} gives \eqref{eq:H-derivative} in relative form. The reflection identity \eqref{eq:HL-reflection} gives \eqref{eq:L-center}. Finally, two differentiations of \eqref{eq:H-uniform} show that $H_n''(p)=O(ne^{-nI_H(p)})$ on $U$. For $|p-\pc|\le n^{-1}$, smoothness gives $I_H(p)=\Kn+O(n^{-1})$, proving \eqref{eq:wedge-second-bound}; reflection handles $L_n$.
\end{proof}

\subsection{Proof of the parametric weighted lattice-corner expansion}\label[appendix]{app:corner-proof}

This subsection proves \Cref{lem:corner} by a parameter-uniform local central limit calculation under exponential tilting, followed by a geometrically weighted corner sum.

\begin{proof}[Proof of \Cref{lem:corner}]
\medskip \noindent\emph{Step 1: exponential tilt and change of measure.} Define the exponentially tilted law
\begin{equation}
 \widetilde\Pp(Z_p=z)
 =\frac{e^{\langle\lambda(p),z\rangle}}
 {M_p(\lambda(p))}\Pp(Z_p=z).
 \label{eq:tilted-law}
\end{equation}
By construction, its mean is zero and its covariance is $\Sigma(p)$. The support hypothesis and positivity imply that $\Sigma(p)$ is positive definite. Compactness of $U$ therefore gives constants $0<c<C<\infty$ such that
\[
 cI_d\preceq\Sigma(p)\preceq CI_d
 \qquad(p\in U).
\]
The change of measure is exact:
\begin{equation}
 \Pp(S_{n,p}=z)
 =e^{-nI(p)}e^{-\langle\lambda(p),z\rangle}
 \widetilde\Pp(S_{n,p}=z).
 \label{eq:exact-change-measure}
\end{equation}
Under the tilted law the rare corner is centered at the mean, while $e^{-nI(p)}$ records its exponential cost under the original law.

\medskip \noindent\emph{Step 2: uniform local central limit estimate.} We first establish a local estimate, uniformly in $p$ and in $\|z\|_\infty\le L\log n$ for every fixed $L$. Let
\[
 \phi_p(u)=\widetilde\Ee
 e^{i\langle u,Z_p\rangle},
 \qquad u\in[-\pi,\pi]^d.
\]
Uniformly in $p\in U$, $\phi_p(u)\to1$ as $u\to0$. Choose $\eta>0$ so that $\phi_p(u)$ lies in the disk $|\zeta-1|<1/2$ whenever $p\in U$ and $\|u\|\le\eta$. The principal branch of $\log\phi_p(u)$ is then $C^{r+2}$ in $p$ and smooth in $u$. Since $\log\phi_p(0)=0$ and $\nabla_u\log\phi_p(0)=i\widetilde\Ee Z_p=0$ for every $p$, Taylor expansion at the origin, also after up to $r$ parameter derivatives, gives
\begin{equation}
 \log\phi_p(u)
 =-\frac12u^{\mathsf T}\Sigma(p)u+R_p(u),
 \label{eq:char-expansion}
\end{equation}
where, for $0\le j\le r$,
\begin{equation}
 \sup_{p\in U,\,\|u\|\le\eta}
 |\partial_p^jR_p(u)|\le C_j\|u\|^3.
 \label{eq:R-bound}
\end{equation}
All constants are uniform because the support is fixed and the masses are smooth on a compact parameter set. The tilted masses in \eqref{eq:tilted-law} are continuous and positive, hence have a common positive lower bound on $U$. For fixed $p$, equality $|\phi_p(u)|=1$ means that $e^{i\langle u,x-y\rangle}=1$ for every pair of support points $x,y$. Since their differences generate $\Zz^d$, this holds only at $u=0\pmod{2\pi\Zz^d}$. Continuity on the compact set
\[
 U\times\{u\in[-\pi,\pi]^d:\|u\|\ge\eta\}
\]
therefore gives a common bound $\rho<1$. The covariance bounds near the origin give, after decreasing $\eta$ if necessary, a constant $c_0>0$ such that
\begin{align*}
 |\phi_p(u)|&\le e^{-c_0\|u\|^2},
 &&\|u\|\le\eta,\\
 |\phi_p(u)|&\le\rho,
 &&u\in[-\pi,\pi]^d,\ \|u\|\ge\eta.
\end{align*}
uniformly in $p$.

Fourier inversion yields
\begin{equation*}
 \widetilde\Pp(S_{n,p}=z)
 =(2\pi)^{-d}\int_{[-\pi,\pi]^d}
 \phi_p(u)^ne^{-i\langle u,z\rangle}\,du.
\end{equation*}
The parameter derivatives may be taken inside this integral. Indeed, repeated differentiation of the finite-support characteristic function and the product $\phi_p(u)^n$ gives, for $0\le j\le r$ and all sufficiently large $n$,
\begin{align*}
 |\partial_p^j\{\phi_p(u)^n\}|
 &\le C_jn^j e^{-c_0(n-j)\|u\|^2},
 &&\|u\|\le\eta,\\
 |\partial_p^j\{\phi_p(u)^n\}|
 &\le C_jn^j\rho^{\,n-j},
 &&\eta\le\|u\|\le\pi\sqrt d.
\end{align*}
For example, every term after $j$ differentiations contains at least $n-j$ undifferentiated factors $\phi_p(u)$, while all derivatives of $\phi_p$ are uniformly bounded. Split the integral into $\|u\|\le n^{-2/5}$ and its complement. The two bounds above show that the complementary part, and each of its first $r$ parameter derivatives, is $O(e^{-c_1n^{1/5}})$. On the small ball put $u=v/\sqrt n$. From \eqref{eq:char-expansion}--\eqref{eq:R-bound}, uniformly for $\|v\|\le n^{1/10}$,
\begin{equation*}
 \phi_p(v/\sqrt n)^n
 =e^{-\frac12v^{\mathsf T}\Sigma(p)v}
 \left[1+O_{C^r(U)}\left(
 n^{-1/2}(1+\|v\|)^{3r+3}e^{C\|v\|^3/\sqrt n}
 \right)\right].
\end{equation*}
For $j=0$, use $|e^{nR}-1|\le |nR|e^{|nR|}$. For $1\le j\le r$, repeated use of the chain rule expresses each parameter derivative of $e^{nR_p(v/\sqrt n)}$ as that exponential times a finite sum of products of factors $n\,\partial_p^aR_p(v/\sqrt n)$. Each term in the difference therefore contains either $e^{nR}-1$ or at least one such factor, both bounded by $O(n^{-1/2}\|v\|^3)$ times a fixed polynomial in $\|v\|$. Parameter derivatives of the Gaussian factor contribute only fixed-degree polynomials in $v$, which are absorbed by a Gaussian envelope. Thus, for all sufficiently large $n$, $0\le j\le r$, and $\|v\|\le n^{1/10}$,
\[
 \sup_{p\in U}
 \left|\partial_p^j\left\{
 \phi_p(v/\sqrt n)^n
 -e^{-v^{\mathsf T}\Sigma(p)v/2}
 \right\}\right|
 \le Cn^{-1/2}(1+\|v\|)^{3r+3}e^{-c\|v\|^2}.
\]
The factor $e^{-i\langle v,z\rangle/\sqrt n}$ causes no parameter-derivative loss because $z$ does not depend on $p$ and $\|z\|\le L\log n$. The displayed bound is integrable after multiplication by every fixed polynomial, so dominated integration applies. Extending the $v$-integral to $\Rr^d$ gives
\begin{equation*}
 \widetilde\Pp(S_{n,p}=z)
 =n^{-d/2}\left\{
 \frac{\exp\left[-z^{\mathsf T}\Sigma(p)^{-1}z/(2n)\right]}
 {(2\pi)^{d/2}\sqrt{\det\Sigma(p)}}
 +O_{C^r(U)}\left(n^{-1/2}(\log n)^{3r+5}\right)\right\}.
\end{equation*}
For $\|z\|_\infty\le L\log n$, the exponential in the first line equals $1+O_{C^r(U)}((\log n)^2/n)$. Thus
\begin{equation}
 \widetilde\Pp(S_{n,p}=z)
 =\frac{n^{-d/2}}
 {(2\pi)^{d/2}\sqrt{\det\Sigma(p)}}
 \left[1+O_{C^r(U)}(\varepsilon_n)\right].
 \label{eq:flat-local-clt}
\end{equation}
Thus each fixed lattice point within logarithmic distance of the tilted mean has mass of order $n^{-d/2}$.

\medskip \noindent\emph{Step 3: geometrically weighted corner sum.} It remains to sum over the corner. Since every coordinate of $\lambda(p)$ is bounded below by a positive constant on $U$, there is $c_2>0$ such that $e^{-\langle\lambda(p),z\rangle}\le e^{-c_2\|z\|_1}$. Choose $L$ sufficiently large. The contribution in \eqref{eq:exact-change-measure} from $\|z\|_\infty>L\log n$ is $O(n^{-A})e^{-nI(p)}$ for an arbitrarily prescribed $A$, uniformly in $p$. More explicitly, for every $A>0$ and $0\le j\le r$, $L$ may be chosen so that
\begin{equation}
 \sup_{p\in U}\left|
 \partial_p^j\left[
 e^{nI(p)}
 \sum_{\substack{z\in\Zz_+^d\\\|z\|_\infty>L\log n}}
 w(z)\Pp(S_{n,p}=z)
 \right]\right|=O(n^{-A}).
 \label{eq:corner-tail-Cr}
\end{equation}
To see this, write $\widetilde p_p(x)$ for the single-trial tilted mass. Its first $r$ log-derivatives are uniformly bounded on the fixed support. For every sample sequence $(x_1,\ldots,x_n)$ and $0\le j\le r$,
\[
 \left|\partial_p^j
 \prod_{\nu=1}^n\widetilde p_p(x_\nu)\right|
 \le C_jn^j\prod_{\nu=1}^n\widetilde p_p(x_\nu).
\]
After summing over sequences with total $z$, this gives
\[
 |\partial_p^j\widetilde\Pp(S_{n,p}=z)|
 \le C_jn^j\widetilde\Pp(S_{n,p}=z).
\]
Derivatives of $e^{-\langle\lambda(p),z\rangle}$ contribute only fixed powers of $\|z\|$. The product rule therefore bounds each derivative of the normalized tilted summand by
\[
 C_j(n+\|z\|)^{r}e^{-c_2\|z\|_1}
 \widetilde\Pp(S_{n,p}=z)
 \le C_j(n+\|z\|)^{r}e^{-c_2\|z\|_1}.
\]
Summing this envelope over $\|z\|_\infty>L\log n$ proves \eqref{eq:corner-tail-Cr} once $L$ is large enough. Choose $A$ larger than $d/2+1$. On the truncated corner, insert \eqref{eq:flat-local-clt}. Termwise parameter differentiation of the absolutely convergent weighted geometric sum is justified by the same uniform exponential envelope. The truncated sum converges in $C^r(U)$ to $W(p)$ with a tail smaller than any chosen power of $n^{-1}$. The geometric weights therefore change the prefactor $W(p)$ but not the rate $I(p)$. This proves \eqref{eq:corner-expansion}.

\medskip \noindent\emph{Step 4: parameter derivatives.} Write
\[
 b(p)=\frac{W(p)}{(2\pi)^{d/2}\sqrt{\det\Sigma(p)}},
 \qquad
 \mathcal T_n(p)=n^{-d/2}e^{-nI(p)}\{b(p)+e_n(p)\},
 \qquad
 \|e_n\|_{C^r(U)}=O(\varepsilon_n).
\]
If $I'(p_0)\ne0$, one differentiation gives
\[
 \mathcal T_n'(p_0)
 =n^{1-d/2}e^{-nI(p_0)}
 \{-I'(p_0)b(p_0)+O(\varepsilon_n+n^{-1})\}.
\]
This proves \eqref{eq:corner-first-derivative}; when $W(p_0)\ne0$, the error can also be written relative to the leading term. If $I'(p_0)=0$, two differentiations give
\[
 \mathcal T_n''(p_0)
 =n^{1-d/2}e^{-nI(p_0)}
 \{-I''(p_0)b(p_0)+O(\varepsilon_n+n^{-1})\}.
\]
When $I''(p_0)>0$ and $W(p_0)>0$, this proves \eqref{eq:corner-second-derivative}.
\end{proof}

\section{Rate separation and localization}\label[appendix]{app:graph-separation}

This appendix provides the graph-rate classification in \Cref{lem:graph-gap} and the uniform estimates used to localize the maximizer and prove uniqueness in \Cref{thm:global-top-t}.

\subsection{Graph-rate separation}\label[appendix]{app:graph-rate-details}

Throughout this appendix, $q=p+\delta$. Recall the Bernoulli relative entropy $\KL$ from \eqref{eq:KL}. For a nonempty edge set $A\subset E(K_{k-t,t})$, let $G_A$ be the bipartite graph formed by its incident vertices and edges. Write $V_L(G_A)$ and $V_H(G_A)$ for its low and high vertex classes, respectively, and $V(G_A)=V_L(G_A)\cup V_H(G_A)$. Define
\begin{equation*}
 \mathcal I_p(G_A)
 =\min_{(x,y)\in\mathcal C(G_A)}
 \left\{
 \sum_{i\in V_L(G_A)}\KL(x_i\|p)
 +\sum_{j\in V_H(G_A)}\KL(y_j\|p+\delta)
 \right\},
\end{equation*}
where
\begin{equation*}
 \mathcal C(G_A)
 =\left\{(x,y)\in[0,1]^{V_L(G_A)}\times[0,1]^{V_H(G_A)}:
 x_i\ge y_j\ \text{for every }(i,j)\in A\right\}.
\end{equation*}
The endpoint convention in \eqref{eq:KL} covers boundary values. On every compact $U\subset(0,1-\delta)$, the objective is jointly continuous and the feasible set is fixed and compact. Hence $p\mapsto\mathcal I_p(G_A)$ is continuous on $U$; strict convexity gives a unique minimizer for every fixed $p\in U$.

\begin{proof}[Proof of \Cref{lem:graph-gap}]
\medskip \noindent\emph{Graph-rate classification.} At the center, write $\ell:=\pc$ and $h:=\pc+\delta=1-\pc$. Parts (a) and (b) are the variational formulas already derived for $I_1(\pc)$ and $I_H(\pc)$, together with reflection for a low-centered wedge. A matching has independent connected components, so graph rates add; this proves (c), including the strict comparison by \Cref{lem:gap}.

For (d), every connected bipartite graph with at least three edges contains a connected two-edge wedge. Since a simple bipartite graph on the three vertices of that wedge has at most those two edges, there exists an edge of the larger graph leaving the wedge vertex set; in particular, at least one new vertex is attached to an existing wedge vertex. The wedge minimizer has a unique interior common level. An added incident vertex cannot remain at its zero-cost natural mean while satisfying the new comparison, so the larger connected graph has a strictly higher rate. The details follow.

Consider first a wedge centered at a high vertex. Its unique rate minimizer assigns the common value
\[
 z_H=\frac{\taud}{1+\taud}
\]
to its two low vertices and its high vertex. The natural means satisfy
\begin{equation*}
 \ell<z_H<h.
\end{equation*}
Indeed, $z_H-\ell=\taud(1-\taud^2)/[(1+\taud)(1+\taud^3)]>0$ and $h-z_H=(1-\taud^4)/[(1+\taud)(1+\taud^3)]>0$. If the leaving edge introduces a low vertex adjacent to the existing high vertex, assigning that new vertex its natural mean $\ell$ violates $\ell\ge z_H$. If it introduces a high vertex adjacent to an existing low vertex, assigning the new vertex its natural mean $h$ violates $z_H\ge h$. Thus the wedge minimizer, augmented by natural means for all new vertices, is not feasible for the larger graph.

For a wedge centered at a low vertex, the common minimizing value is $z_L=1/(1+\taud)$, and the same calculation gives $\ell<z_L<h$. The identical argument applies. Let $G_{\mathrm w}$ denote the chosen wedge. Restricting any feasible point for the larger graph $G$ to $G_{\mathrm w}$, and using nonnegativity of the remaining relative-entropy terms, gives
\[
 \mathcal I_{\pc}(G)\ge \mathcal I_{\pc}(G_{\mathrm w})=\Kn.
\]
Equality would force the restriction to $G_{\mathrm w}$ to equal its unique minimizer and every new vertex to equal its natural mean. The preceding infeasibility rules this out. Hence $\mathcal I_{\pc}(G)>\Kn$.

It remains to pass from connected components to every edge set in $\mathcal R_n$. Writing $\operatorname{Comp}(G_A)$ for the set of connected components of $G_A$, we have
\[
 \mathcal I_{\pc}(G_A)
 =\sum_{C\in\operatorname{Comp}(G_A)}\mathcal I_{\pc}(C).
\]
If an edge set with at least three edges has a connected component with at least three edges, part (d), together with additivity over disjoint components, gives total rate strictly larger than $\Kn$. If all components have at most two edges and one component is a wedge, the remaining components contribute at least $\Jn$, so the total rate is at least $\Kn+\Jn>\Kn$. If every component is a single edge, the set is a matching of at least three edges and has rate at least $3\Jn>\Kn$. Finally, two or more disjoint wedges have rate at least $2\Kn>\Kn$. These cases exhaust all edge sets containing at least three edges.

\medskip \noindent\emph{Uniform probability and derivative bounds.} There are only finitely many edge sets in $K_{k-t,t}$. The continuity of $\mathcal I_p(G_A)$ proved above implies that all strict center gaps persist, after shrinking $U$, with a common margin $2\eta>0$. For any fixed graph event, ignore the auxiliary uniforms and retain the weak count inequalities $X_i/n\ge Y_j/n$. There are at most $(n+1)^{|V(G_A)|}$ vectors of empirical proportions, and the method of types gives
\[
 \Pp\left(\bigcap_{e\in A}E_e\right)
 \le(n+1)^{|V(G_A)|}e^{-n\mathcal I_p(G_A)}.
\]
The auxiliary tie variables only multiply boundary count configurations by numbers in $[0,1]$; they cannot change the exponential upper rate. This proves \eqref{eq:remainder-prob-bound}.

On $U$, all Bernoulli parameters are bounded away from zero and one. For a fixed $G_A$, abbreviate $V_L=V_L(G_A)$ and $V_H=V_H(G_A)$. For count vectors $r=(r_i)_{i\in V_L}$ and $s=(s_j)_{j\in V_H}$, let $\mathcal L_p(r,s)$ denote their joint binomial likelihood. Its score is
\begin{equation*}
 \mathcal S_p(r,s)
 :=\partial_p\log \mathcal L_p(r,s)
 =\sum_{i\in V_L}\frac{r_i-np}{p(1-p)}
 +\sum_{j\in V_H}\frac{s_j-nq}{q(1-q)}.
\end{equation*}
Moreover,
\[
 \partial_p\mathcal S_p(r,s)
 =-\sum_{i\in V_L}\left\{\frac{r_i}{p^2}
       +\frac{n-r_i}{(1-p)^2}\right\}
  -\sum_{j\in V_H}\left\{\frac{s_j}{q^2}
       +\frac{n-s_j}{(1-q)^2}\right\}.
\]
Since the graph is fixed,
\[
 |\mathcal S_p(r,s)|\le Cn|V(G_A)|,\qquad
 |\partial_p\mathcal S_p(r,s)|\le Cn|V(G_A)|
\]
uniformly in $p\in U$ and in the counts. Consequently
\[
 |\partial_p\mathcal L_p(r,s)|\le Cn\mathcal L_p(r,s),\qquad
 |\partial_p^2\mathcal L_p(r,s)|
 =\mathcal L_p(r,s)|\mathcal S_p(r,s)^2+\partial_p\mathcal S_p(r,s)|
 \le Cn^2\mathcal L_p(r,s).
\]
The tie weight attached to a count vector lies in $[0,1]$ and is independent of $p$. Differentiating the finite count sum and then applying the preceding type bound therefore adds at most the factors $Cn$ and $Cn^2$. These polynomial factors are absorbed by reducing $2\eta$ to $\eta$, proving \eqref{eq:R-C2}.

\medskip \noindent\emph{Local edge-dominance gap.} For the final local-gap assertion, every occurring non-single-edge graph has center rate at least $\Kn>\Jn=I_1(\pc)$. There are finitely many such graphs, and their rates and $I_1$ are continuous. Hence, after shrinking $U$, a common $\gamma>0$ satisfies
\[
 \mathcal I_p(G)\ge I_1(p)+\gamma
\]
for every non-single-edge graph occurring in \eqref{eq:exact-graph-decomp}. Summing the likelihood-score bounds above over these finitely many edge sets gives the stated derivative bound.
\end{proof}

\subsection{Uniform localization and curvature estimates}\label[appendix]{app:localization-details}

We give the uniform estimates used in the proof of \Cref{thm:global-top-t}.

\medskip \noindent\emph{Step 1: global localization.} The rate $I_1$ has its unique minimum $\Jn$ at $\pc$. Hence, for every fixed $\varepsilon>0$, there is $c_\varepsilon>0$ such that
\begin{equation*}
 I_1(p)\ge\Jn+c_\varepsilon
 \qquad\text{when }|p-\pc|\ge\varepsilon.
\end{equation*}
The union bound and method of types give
\begin{equation}
 Q_n(p)
 \le t(k-t)\Pp(X\ge Y)
 \le t(k-t)(n+1)^2e^{-nI_1(p)}.
 \label{eq:Q-away-upper}
\end{equation}
This bound is valid on the full closed interval: at $p=0$ or $p=1-\delta$ it is interpreted with the extended KL convention in \eqref{eq:KL}, or obtained by continuity from the interior.

At the center, the lower Bonferroni bound together with \Cref{lem:one-edge-rate,lem:wedge,lem:gap} shows that all pair intersections are $o(a_n(\pc))$. Therefore
\begin{equation}
 Q_n(\pc)
 \ge\frac{t(k-t)}{2}A_{1,\delta}n^{-1/2}e^{-n\Jn}
 \label{eq:Q-center-lower}
\end{equation}
for all sufficiently large $n$. Comparing \eqref{eq:Q-away-upper} and \eqref{eq:Q-center-lower} puts every global maximizer of $Q_n$ inside $(\pc-\varepsilon,\pc+\varepsilon)$ eventually. Taking $\varepsilon$ small enough, the same argument puts all global maximizers in the interval $U$ fixed after \Cref{lem:graph-gap}.

\medskip \noindent\emph{Step 2: derivative signs outside an $n^{-1}$ neighborhood.} Write the uniform one-edge expansion as
\begin{equation}
 a_n(p)
 =n^{-1/2}e^{-nI_1(p)}\{A_1(p)+\Delta_n(p)\},
 \qquad
 \|\Delta_n\|_{C^2(U)}=o(1).
 \label{eq:a-uniform}
\end{equation}
By \Cref{lem:one-edge-rate}, both $I_1$ and $A_1$ are even about $\pc$. After shrinking $U$, there are constants $c_1,c_2>0$ such that
\begin{equation*}
 \operatorname{sgn}I_1'(p)=\operatorname{sgn}(p-\pc),
 \qquad
 |I_1'(p)|\ge c_1|p-\pc|,
 \qquad
 |A_1'(p)|\le c_2|p-\pc|.
\end{equation*}
Differentiating \eqref{eq:a-uniform} gives
\begin{equation*}
 \frac{a_n'(p)}{n^{-1/2}e^{-nI_1(p)}}
 =-nI_1'(p)\{A_1(p)+\Delta_n(p)\}
   +A_1'(p)+\Delta_n'(p).
\end{equation*}
Let $A_*=\inf_{p\in U}A_1(p)>0$. For all large $n$, $A_1+\Delta_n\ge A_*/2$ on $U$. If $n^{-1}\le|p-\pc|$, then
\[
 n|I_1'(p)|\{A_1(p)+\Delta_n(p)\}
 \ge \frac{c_1A_*}{2}n|p-\pc|
 \ge \frac{c_1A_*}{2},
\]
whereas
\[
 \frac{|A_1'(p)|}
 {n|I_1'(p)|\{A_1(p)+\Delta_n(p)\}}
 \le\frac{2c_2}{nc_1A_*},
 \qquad
 \frac{\|\Delta_n'\|_\infty}
 {n|I_1'(p)|\{A_1(p)+\Delta_n(p)\}}
 \le\frac{2\|\Delta_n'\|_\infty}{c_1A_*}.
\]
Both ratios vanish uniformly. Hence, uniformly for $n^{-1}\le|p-\pc|\le\operatorname{diam}(U)$,
\begin{equation*}
 \operatorname{sgn}a_n'(p)=-\operatorname{sgn}(p-\pc)
\end{equation*}
for all sufficiently large $n$. The same estimate gives the quantitative lower bound
\begin{equation*}
 |a_n'(p)|\ge c_4 n^{-1/2}e^{-nI_1(p)}
\end{equation*}
on this region, for a constant $c_4>0$ independent of $n$.

By \Cref{lem:graph-gap}, the derivative of the sum of all non-single-edge terms is at most a polynomial factor times $e^{-n(I_1(p)+\gamma)}$ and is therefore $o(|a_n'(p)|)$ uniformly by the preceding lower bound. Consequently
\begin{equation}
 Q_n'(p)>0\quad\text{on }U\cap(-\infty,\pc-n^{-1}],
 \qquad
 Q_n'(p)<0\quad\text{on }U\cap[\pc+n^{-1},\infty)
 \label{eq:Q-prime-outer-sign}
\end{equation}
for all sufficiently large $n$.

\medskip \noindent\emph{Step 3: strict concavity in the central interval.} For $|p-\pc|\le n^{-1}$, two differentiations of \eqref{eq:a-uniform} give
\begin{align*}
 a_n''(p)
 =n^{-1/2}e^{-nI_1(p)}
 \Bigl(&[n^2I_1'(p)^2-nI_1''(p)]\{A_1(p)+\Delta_n(p)\}
 \\
 &-2nI_1'(p)\{A_1'(p)+\Delta_n'(p)\}
 +A_1''(p)+\Delta_n''(p)\Bigr).
\end{align*}
Here $I_1'(p)=O(n^{-1})$, while $I_1''(p)$ is bounded below by a positive constant. The term $-nI_1''(p)A_1(p)$ dominates, and
\begin{equation*}
 a_n''(p)\le-c_3n^{1/2}e^{-n\Jn}
 \qquad(|p-\pc|\le n^{-1})
\end{equation*}
for some $c_3>0$ and all large $n$.

More precisely, the expansion above is uniform on the central interval: $I_1(p)=\Jn+O(n^{-2})$, $I_1''(p)=I_1''(\pc)+O(n^{-1})$, and $A_1(p)=A_{1,\delta}+O(n^{-1})$. Thus
\begin{equation*}
 a_n''(p)
 =-I_1''(\pc)A_{1,\delta}n^{1/2}e^{-n\Jn}\{1+o(1)\}
\end{equation*}
uniformly for $|p-\pc|\le n^{-1}$.

By \eqref{eq:wedge-second-bound}, two derivatives of a wedge term are $O(ne^{-n\Kn})$. Two derivatives of the disjoint-pair term are $O(e^{-2n\Jn})$, and \eqref{eq:R-C2} controls the remainder. Since $n^{1/2}e^{-n(\Kn-\Jn)}\to0$, all these contributions are $o(n^{1/2}e^{-n\Jn})$. Combining these bounds with the preceding estimate proves \eqref{eq:uniform-central-curvature}, and hence
\begin{equation*}
 Q_n''(p)<0
 \qquad(|p-\pc|\le n^{-1})
\end{equation*}
for all sufficiently large $n$.

The derivative signs in \eqref{eq:Q-prime-outer-sign} and strict concavity on the central interval imply that $Q_n$ has exactly one maximizer in $U$. Step~1 makes it the unique global maximizer, proving eventual uniqueness.

\end{document}